\documentclass[11pt,oneside,reqno]{amsart}
\usepackage{amsmath,amssymb,amsbsy,amsfonts,amsthm,
latexsym,amsopn,amstext, amsxtra,euscript,amscd,hyperref,enumitem}
\usepackage{xcolor}
\newtheorem{theorem}{Theorem}[section]
\newtheorem{definition}[theorem]{Definition}
\newtheorem{proposition}[theorem]{Proposition}

\newtheorem{remark}[theorem]{Remark}

\newtheorem{lemma}[theorem]{Lemma}

\def \C{\mathbb{C}}

\def \N{\mathbb{N}}

\def \Q{\mathbb{Q}}
\def \R{\mathbb{R}}
\def \Z{\mathbb{Z}}
\begin{document}
	
	\title{Algebraic approximations to linear combinations of $S$-units}
	\author{Parvathi S Nair, Veekesh Kumar and S. S. Rout}%
	\address[Parvathi S Nair]{Department of Mathematics, National Institute of Technology Calicut, 
		Kozhikode-673 601, India.}
	\email[]{parvathisnair60@gmail.com, parvathi\_p220245ma@nitc.ac.in}
	\address[Veekesh Kumar]{Department of Mathematics, Indian Institute of Technology Dharwad, Chikkamalligawad village, Dharwad, Karnataka 580007, India}
	
	\email[]{veekeshk@iitdh.ac.in}
	\address[S. S. Rout]{Department of Mathematics, National Institute of Technology Calicut, 
		Kozhikode-673 601, India.}
	\email[]{sudhansu@nitc.ac.in}
	\subjclass[202 0] {Primary 11J68; Secondary 11J87, 11R06}
	\bigskip
	\begin{abstract} 
		Let $\Gamma\subset \bar{\Q}^{\times}$ be a finitely generated multiplicative group of algebraic numbers, let $\alpha_1,\ldots,\alpha_m$ be non-zero algebraic numbers, and let $\varepsilon >0$ be fixed. In this paper, we prove that there exist only finitely many tuples $(u_1, \ldots, u_m, q, p)\in \Gamma^m\times\mathbb{Z}^2$ with $d = [\mathbb{Q}(u_1, \ldots, u_m):\mathbb{Q}]$ such that for any two tuples $(u_1,\ldots,u_m)$ and $(u'_1,\ldots,u'_m)$, we have $\frac{u_{i_1}}{u_{i_2}}\neq \frac{u'_{i_1}}{u'_{i_2}}$ for $1\leq i_1\neq i_2\leq m$ and it is stable under Galois conjugation over $\Q$, $\max\{|\alpha_1 qu_1|, \ldots, |\alpha_m qu_m|\}>1$, the tuple $(\alpha_1qu_1, \ldots, \alpha_mq u_m)$ is not pseudo-Pisot and \[0< \left|\sum_{i=1}^m \alpha_iq u_i - p\right|<\frac{1}{\left(\prod_{i=1}^mH( u_i)\right)^{\varepsilon} |q|^{md+\varepsilon}},\] where $H(u_i)$ denotes the absolute Weil height. This result extends one of the main results of Corvaja-Zannier \cite{corv}. In addition, we prove a result similar to \cite[Theorem 1.4]{kul} in a more general setting. In our proofs, we exploit the subspace theorem based on the work of Corvaja-Zannier.
	\end{abstract}
	\maketitle
	
	\section{introduction}
	Throughout this paper, $\|x\|$ will denote the distance of the real number $x$ from the nearest integer, that is, 
	\[\|x\| = \min\{|x-n|: n\in \mathbb{Z}\}.\]
	For a given real number $\alpha>1$, the distribution of the sequence $(\|\alpha^n\|)_{n\geq 1}$ is less known. For example, we do not know whether the sequence $(\|(3/2)^n\|)_{n\geq 1}$ is dense in $[0,1/2)$. In 1957, Mahler \cite{mah} proved that if $\alpha=k/\ell$ is a rational number greater than $1$ and not an integer and $\varepsilon$ a positive real number, then the set $\{n\in \N:\ \|\alpha^n\|<\ell^{-\varepsilon n}\}$ is finite. He used a $p$-adic extension of Roth's theorem, due to Ridout \cite{ridd1}. At the end of his paper, Mahler pointed out that the conclusion does not hold if $\alpha$ is the golden ratio. Actually, it is easily seen that a counterexample is provided by any Pisot number. Recall that a Pisot number is a real algebraic integer greater than $1$  all of whose Galois conjugates (except itself) have an absolute value less than $1$. He also asked which algebraic numbers have the same property as the rationals in the theorem.
	
	In the remarkable paper \cite{corv}, Corvaja-Zannier completely solved this question. Precisely, they proved the following: {\em Let $\alpha>1$ be a real algebraic number, and let $0<\ell<1$. Suppose that $\{n\in \N:\ \|\alpha^n\|<\ell^n\}$ is infinite. Then, there is a positive integer $d\geq 1$ such that $\alpha^d$ is a Pisot number}. Their proof rests on a clever application of the Schmidt's subspace theorem. They also obtain a Thue-Roth inequality with ``moving target" which deals with elements from a finitely generated multiplicative group of algebraic numbers. The authors showed that the conclusion of their theorem is not generally true without assuming that $\alpha$ is algebraic. Note that some related results on integer values of certain expressions involving linear recurrence sequences were obtained in \cite{corv1998, corv2002a, corv2002b}. The simultaneous approximation of algebraic numbers has also been studied in the same spirit of Corvaja and Zannier (see, for example, \cite{vek22, vek23}). 
	\smallskip
	
	In $2019$, Kulkarni et al. \cite{kul} extended the result of Corvaja and Zannier \cite[Theorem 1]{corv}, that is, to linear combinations of powers of algebraic numbers and to linear recurrence sequences (note that $(\alpha^n)_{n\geq1}$ is a linear recurrence sequence). Recently, Kumar and Prasad \cite{vp25} strengthened \cite[Theorem 1]{corv} and extended \cite[Theorem 1.4]{kul} to a more general sequence of the form $\{\lambda q_1\alpha_1^n +\cdots+ \lambda q_k\alpha_k^n: n\in \N\}$, where $q_i$ for $i=1, \ldots,k$ are algebraic numbers and $\lambda\in \N$. 
	
	The above mentioned results are not effective since Mahler's proof relies on Roth's theorem, and Corvaja-Zannier and Kulkarni et al. results relies on Schmidt's subspace theorem. However, Baker and Coates \cite{bakercoates} studied the effective version of Mahler's results \cite{mah} using the theory of linear forms in $p$-adic logarithms, and Bugeaud \cite[Theorem 5]{bugeaud} studied the effective version of Corvaja and Zannier's result \cite[Theorem 1]{corv}. 
	\smallskip
	

Let $\bar{\Q}$ be the algebraic closure of ${\Q}$ in ${\C}$.
Now to state our result,  we need the following definitions.
Let $\mu$ denote the group of roots of unity. For a tuple of non-zero algebraic numbers $(\beta_1, \ldots, \beta_n)$, set
\[P:=\{\beta\in \bar{\Q}^{\times}\setminus\{\beta_1, \ldots, \beta_n\}\mid  \beta = \sigma(\beta_i) \;\mbox{for some}\;\; \sigma\in \mbox{Gal}(\bar{\Q}/\Q), i=1, \ldots, n\}.\]
We call a tuple $(\beta_1, \ldots, \beta_n)$ is {\em pseudo-Pisot} if $\beta_1, \ldots, \beta_n$ are pairwise distinct, $$\sum_{i=1}^n\beta_i+\sum_{\beta\in P} \beta \in \Z$$ and $|\beta|<1$ for every $\beta\in P$. 
	
	
One of the main results of this paper is as follows.
\begin{theorem}\label{maintheorem1} 
Let $\Gamma\subset \bar{\Q}^{\times}$ be a finitely generated multiplicative group of algebraic numbers, let $\alpha_1,\ldots,\alpha_m$ be non-zero algebraic numbers, and let $\varepsilon >0$ be fixed. Then there are only finitely many tuples $(u_1, \ldots, u_m, q, p)\in \Gamma^m\times\mathbb{Z}^2$ with $d = [\mathbb{Q}(u_1, \ldots, u_m):\mathbb{Q}]$ such that 
\begin{enumerate}[label=\textnormal{(\roman*)}]
\item for any two tuples $(u_1,\ldots,u_m)$ and $(u'_1,\ldots,u'_m)$, we have $\frac{u_{i_1}}{u_{i_2}}\neq \frac{u'_{i_1}}{u'_{i_2}}$ for $1\leq i_1\neq i_2\leq m$  and it is stable under Galois conjugation over $\Q$;
\item $\displaystyle\max_{\substack{1\leq i\leq m}} \{|\alpha_i qu_i|\}>1$;
\item the tuple $(\alpha_1qu_1, \ldots, \alpha_mq u_m)$ is not pseudo-Pisot; and
			\item \[0< \left|\sum_{i=1}^m \alpha_iq u_i - p\right|< \frac{1}{(\prod_{i=1}^mH(u_i))^{\varepsilon} |q|^{md+\varepsilon}}.\]
		\end{enumerate}
	\end{theorem}
	
	\begin{remark}
		When we put $m=1$ and $\alpha = \delta$ in Theorem \ref{maintheorem1}, we recover the main Theorem of Corvaja and Zannier in \cite[p.177]{corv}.
		Theorem \ref{maintheorem1} provides a rational approximation to the linear combinations of $S$-units.
	\end{remark}
	In the following result, we establish a result analogous to \cite[Theorem 1.4]{kul} in a more general setting, where we consider elements from a finitely generated multiplicative group in place of the power of algebraic numbers. First, we need the following definitions. Let $\rho, \lambda  \in \bar {\mathbb{Q}}^{\times}$, define an equivalence relation as follows: \[\rho \sim \lambda \;\;\mbox{if there is a}\;\; \sigma \in \mbox{Gal}(\bar{\mathbb{Q}}/\mathbb{Q})\;\; \mbox{such that}\;\;\frac{\rho}{\sigma(\lambda)} \in \mu.\]
    \begin{definition}
	We call a tuple $(\beta_1, \ldots, \beta_n)$ satisfies the property
	\begin{itemize}
		\item  (P1) for $1\leq i \leq n$, if $\rho \neq \beta_i$ is a Galois conjugate to $\beta_i$ over $\Q$, then  $\frac{\rho}{\beta_i} \notin \mu$;
		\item (P2) for $1\leq i,j \leq n$, if $\beta_i \sim \beta_j$, then $\beta_i$ is Galois conjugate to $\beta_j$ over $\Q$.
	\end{itemize}
	 \end{definition}
	\begin{theorem}\label{thm2}
		Let $\Gamma\subset \bar{\Q}^{\times}$ be a finitely generated multiplicative group of algebraic numbers,  let $\alpha_1,\ldots,\alpha_m$ be non-zero algebraic numbers and let $\varepsilon_1>0$ be given. Let $\mathcal{N}_{1}'$ be an infinite set of tuples $(u_1, \ldots, u_m)\in \Gamma^m$ such that $|u_i|\geq 1$ for $1\leq i \leq m$ and for any two tuples $(u_1,\ldots,u_m)$ and $(u'_1,\ldots,u'_m)$, we have $\frac{u_{i}}{u_j}\neq \frac{u'_i}{u'_j}$ with $1\leq i\neq j\leq m$. Furthermore, each tuple $(u_1, \ldots, u_m)$ satisfies both properties $(P1)$, $(P2)$ and
		\begin{equation}\label{eq.1.01}
		0<\left\|\sum_{i=1}^m \alpha_i u_i \right\|< \frac{1}{(\prod_{i=1}^mH(u_i))^{\varepsilon_1}}.
		\end{equation}
		Then there exists an infinite subset of $\mathcal{N}_{1}'$ such that the following holds:
\begin{enumerate}[label=\textnormal{(\roman*)}]
\item There exists a positive constant  $C\geq 1$ such that
$$
\max\{|u_1|_v,\ldots,|u_m|_v\}<C
$$
 holds for all non-Archimedean places $v$.
\item For each $\sigma\in \mbox{Gal}(\bar{\Q}/\Q)$ and $i=1, \ldots, m$ such that $\frac{\sigma(u_i)}{u_j} \notin \mu $ for $j=1,\ldots,m$, we have $|\sigma(u_i)|<1$.
\item The tuple $(\alpha_1 u_1,\ldots, \alpha_m u_m)$ is pseudo-Pisot.
\item  If $\frac{\sigma(u_i)}{u_j} \in \mu$ for $(\sigma, i, j) \in \mbox{Gal}(\bar{\Q}/\Q) \times \{1,\ldots,m\}^2$, then 	$\sigma(\alpha_iu_i)=\alpha_ju_j$. Moreover, if  $\sigma(\alpha_iu_i)=\alpha_ju_j$, then $\frac{\sigma(u_i)}{u_j}=\frac{\sigma(u'_i)}{u'_j}$ for any two tuples $(u_1,\ldots,u_m)$ and $(u'_1,\ldots,u'_m)$. 
\end{enumerate}
\end{theorem}

 \begin{remark}
 In particular if we take $m=1$, $\alpha_1=1$, $u_1=\beta^n$ for some real algebraic number $\beta>1$ and $n\in \N$, properties (i) and (iii) of Theorem \ref{thm2} recover the first theorem of  Corvaja and Zannier in \cite[p.176]{corv}.
	\end{remark}
	The organization of this paper is as follows. In the next section, we present some preliminary results concerning the Weil height and the technical ingredients needed to prove our results, including the famous Schmidt’s subspace theorem and some of its applications. In Section \ref{sec3}, we present the proof of Theorem \ref{maintheorem1} following the methods of the papers \cite{corv, vek22} with suitable modifications. Finally, in Section \ref{sec4}, we prove Theorem \ref{thm2} by adapting \cite{kul}. 
	
	\section{Preliminaries}\label{sec2}
Let $K\subset \mathbb{C}$ be a number field which is Galois over $\mathbb{Q}$ with the Galois group $\mbox{Gal}(K/\mathbb{Q})$. Let $M_K$ be the set of places of $K$ and $M_K^{\infty}$ be the set of all archimedean places of $K$ and $M_K^0 = M_K\setminus M_K^{\infty}$. For each place $w\in M_K$, let $K_w$ denote the completion of the number field $K$ with respect to $w$ and $d(w)=[K_w:\mathbb{Q}_v]$, where $v$ is the restriction  of $w$ to $\mathbb{Q}$. For every $w\in M_K$ whose restriction on $\mathbb{Q}$ is $v$ and $x\in K$, we define the normalized absolute value $|\cdot |_w$ as follows:
$$|x|_w:=|\mbox{Norm}_{K_w/\mathbb{Q}_v}(x)|_v^{{1/[K:\mathbb{Q}]}}.$$
In this notation, we have the product formula
\[\displaystyle\prod_{\omega\in M_K}|x|_\omega=1,\] for all $x\in K^\times$.
Let $S$ be a finite set of places of $K$ containing $M_K^{\infty}$. The ring of $S$-integers is defined as 
\[\mathcal{O}_{S} := \{u\in K: |u|_v \leq  1  ~~ \forall v \notin S\}\]
and the group of $S$-units is defined as
\[ \mathcal{O}_{S}^{\times} := \{u\in K: |u|_v = 1  ~~ \forall v \notin S\}.\] Note that for a number field $K$, we can choose a suitable finite set of places $S$, containing archimedean places and the corresponding group of $S$-units stable under Galois conjugation. The absolute Weil height $H(x)$ is defined as
$$
H(x):=\prod_{\omega\in M_K}\mbox{max}\{1,|x|_\omega\} \mbox{ for all } x\in K.
$$
It is evident that this height is independent of the choice of the number field $K$ that contains $x$. For a vector $\textbf{x} =(x_1,\ldots,x_n)\in K^n \backslash \{0\}$  and for a place $\omega\in M_K$, the $\omega$-norm for  $\textbf{x}$, denoted by $\|\textbf{x}\|_\omega$, is defined by 
	$$
	\Vert\textbf{x}\Vert_\omega:=\mbox{max}\{|x_1|_\omega,\ldots,|x_n|_\omega\}
	$$
	and  the projective height,  $H(\textbf{x})$, is defined by 
	$$
	H(\textbf{x})=\prod_{\omega\in M_K}\Vert\textbf{x}\Vert_\omega.
	$$
Now we are ready to state a more general version of the Schmidt's subspace theorem, which was formulated by Evertse and Schlickewei (see \cite[Chapter 7]{bomb},  \cite[Theorem $1D^\prime$]{schmidt} and  \cite[Theorem II.2]{zannier}).
\begin{theorem}[Subspace Theorem] \label{schli}
Let $K$ be a number field and $n \geq 2$ be an integer. Let $S$ be a finite set of places of $K$ that contains all the archimedean places of $K$.  For each $v \in S$, let $L_{v,1}, \ldots, L_{v,n}$ be linearly independent linear  forms in variables $X_1,\ldots,X_n$ with coefficients in $K$. For any $\varepsilon>0$, the set of solutions $\mathbf{x} \in K^n\backslash \{0\}$ of the inequality 
\begin{equation*}
\prod_{v\in S}\prod_{i=1}^n \frac{|L_{v, i}(\mathbf{x})|_v}{\Vert\mathbf{x}\Vert_v} \leq \frac{1}{H(\mathbf{x})^{n+\varepsilon}}
\end{equation*}
lies in finitely many proper subspaces of $K^n$.
\end{theorem}
The next lemma is useful to prove Theorem \ref{thm2}.
\begin{lemma}\label{lem2.4}
Let $n \in \mathbb{N}$, let $K$ be a number field, let $S$ be a finite set of places of $K$ containing all archimedean places, and let $\lambda_1,\ldots, \lambda_n$ be non-zero elements of $K$. Fix $\mathit{v}\in S$ and $\varepsilon>0$. Let ${\bf y} := (y_1,\ldots,y_n) \in (\mathcal{O}_S^\times)^n$ and $\mathcal R$ be defined as
\begin{equation*}
\begin{split}
\mathcal R: =\left\{(y_1,\ldots,y_n): \left|\sum_{j=1}^{n} \lambda_j y_j\right|_{\mathit{v}} \leq \frac{\max\{|y_1|_{\mathit{v}},\ldots,|y_n|_{\mathit{v}}\}}{H({\bf y})^{\varepsilon}}\right\}.
\end{split}
\end{equation*}
If $\mathcal R$ is infinite, then there exists a non-trivial linear relation of the form $$c_1 y_1+\cdots+c_n y_n=0,$$ with $c_i \in K$ for $1 \leq i \leq n$ satisfied by infinitely many elements of $\mathcal R$.
\end{lemma}
\begin{proof}
See \cite[Proposition 2.3]{kul}.
\end{proof}
We also need the following lemma, which is a special case of the $S$-unit equation theorem proved by Evertse.  
\begin{lemma}[\cite{ev1984}]\label{lem2.2}
Let $K$ be a number field and let $S$  be a finite set of places of $K$ containing all the archimedean places, with group of $S$-units $\mathcal{O}_S^{\times}$.  Let $a_1,\ldots, a_m$ be non-zero elements of $K$. Let $\mathfrak{X}\subset (\mathcal{O}_S^\times)^m$ be an infinite set of solutions to the equation
\begin{equation}\label{uniteq}
a_1y_1+\cdots+a_my_m=0.
\end{equation}
Then there exist two distinct indices $i\neq j$, two non-zero elements $a,b \in K^{\times}$ and an infinite subset  $\mathfrak{X'}\subset \mathfrak{X}$ of $\mathfrak{X}$ such that for all $(y_1,\ldots, y_m) \in \mathfrak{X'}$,
$$ay_i+by_j=0.$$
\end{lemma}
\section{Proof of Theorem \ref{maintheorem1}}\label{sec3}
\subsection{Key result}
At first we will prove the following proposition, which plays an essential role in proving Theorem \ref{maintheorem1}.
\begin{proposition}\label{prop1}
Let $K$ be a number field of degree $n$ which is Galois over $\mathbb{Q}, k\subset K$  be a subfield of $K$ of degree $d$ over $\mathbb{Q}$ and $\alpha_1,\ldots,\alpha_m\in K$ be given non-zero algebraic numbers. Let $S$  be a finite set of places on $K$  such that $S$ contains all the archimedean places of $K$. Let $\varepsilon >0$ be given. Suppose that we have an infinite set $\mathcal{B}$ of points $(u_1,\ldots, u_m, q, p) \in (\mathcal{O}_S^\times \cap k)^m\times\mathbb{Z}^2$
such that 
\begin{enumerate}[label=\textnormal{(\roman*)}]
\item for any two tuples $(u_1,\ldots,u_m)$ and $(u'_1,\ldots,u'_m)$, we have $\frac{u_{i_1}}{u_{i_2}}\neq \frac{u'_{i_1}}{u'_{i_2}}$ for $1\leq i_1\neq i_2\leq m$ and it is stable under Galois conjugation over $\Q$;
\item $\max \{|\alpha_1 qu_1|, \ldots, |\alpha_m qu_m|\}>1$;
\item the tuple $(\alpha_1qu_1, \ldots, \alpha_mq u_m)$ is not pseudo-Pisot; and 
\item \begin{equation}\label{prop1eq1}
0< \left|\sum_{i=1}^m \alpha_iq u_i - p\right| < \frac{1}{(\displaystyle\prod_{i=1}^mH(u_i))^{\varepsilon}|q|^{md+\varepsilon}}.
\end{equation}  
\end{enumerate}
Then there exists a proper subfield $k'\subset k$, a non-zero tuple $(u_1^{(1)}, \ldots, u_m^{(1)})\in k^m$ and an infinite subset $\mathcal{B}'\subset \mathcal{B}$  such that for all tuples $(u_1,\ldots, u_m, q, p) \in\mathcal{B}'$, we have \[\left(\frac{u_1}{u_1^{(1)}}, \ldots, \frac{u_m}{u_m^{(1)}}\right)\in (k')^m.\]
\end{proposition}
\begin{remark}
Note that if $K=\Q$, then $\mathcal{B}$ is a finite set since the field of  rational numbers does not admit proper subfields.
\end{remark}	
\subsection{Proof of Proposition \ref{prop1}}
First, observe that along an infinite subset of $\mathcal{B}$, at least one of $H(u_i)$ cannot be fixed. Suppose that $H(u_i)$ is fixed for all $i$ with $1\leq i\leq m$, then the projective height $H({\bf u})$ is also fixed, where ${\bf u}:= (u_1, \ldots, u_m)$. Thus, by  the Northcott property, there are only finitely many such tuples $(u_1,\ldots,u_m)$. There exists an infinite subset $\mathcal{A}$ of $\mathcal{B}$ such that the tuple $(u_1,\ldots,u_m)$ is constant along the set  $\mathcal{A}$, let us say $u_i=u_{0,i}$ for $i=1,\ldots,m$, where $u_{0,i}$ are some fixed algebraic numbers in $K$. Using \eqref{prop1eq1} and the fact that $H({\bf u})\leq \prod_{i=1}^m H(u_i)$, the inequality 
$$
0<\left|\alpha_1  u_{0,1}+\cdots+\alpha_m  u_{0,m}-\frac{p}{q}\right|<\frac{1}{(\displaystyle\prod_{i=1}^mH(u_i))^{\varepsilon}q^{md+1+\varepsilon}}\leq \frac{1}{H({\bf u})^{\varepsilon}q^{md+1+\varepsilon}}
$$
has infinitely many solutions in the rationals $p/q$, which contradicts Roth's theorem (\cite[Chapter II]{schmidt}). Hence, at least one of $H(u_i)$ cannot be fixed along infinitely many tuples $(u_1,\ldots,u_m, q, p) \in \mathcal{B}$.
Given that $K$ is a number field which is Galois over $\mathbb{Q}$ and $k\subset K$  be a subfield of degree $d$  over $\mathbb{Q}$. Since $\mathbb{Q}$ does not admit any proper subfield in it, we can take $d\geq 2$.  Let $\mathcal{G}:=\mbox{Gal}(K/\mathbb{Q})$ be the Galois group of $K$ over $\mathbb{Q}$ and let $\mathcal{H}:=\mbox{Gal}(K/k)\subset \mathcal{G}$  be the subgroup of $\mathcal{G}$ fixing $k$.  
Since $K$ is Galois over $\mathbb{Q}$, we have that $K$ is Galois over $k$ and $|\mathcal{G}/\mathcal{H}| =d$.  Therefore among the $n$ embeddings of $K$, there are $d$ embeddings $\sigma_1,\ldots,\sigma_d$ which are the complete set of representatives of the left cosets of $\mathcal{H}$ in $\mathcal{G}$, with $\sigma_1$ as the identity.  More precisely, we have 
$$
\mathcal{G}/\mathcal{H}:=\{\mathcal{H}, \sigma_2 \mathcal{H},\ldots,\sigma_d \mathcal{H}\}.
$$
For each $\rho\in \mathcal{G}$, we define an archimedean absolute value on $K$ by the formula
\begin{equation}\label{eq3.23}
|x|_\rho:=|\rho^{-1}(x)|^{d(\rho)/[K:\mathbb{Q}]},
\end{equation}
where $d(\rho)=1$ if $\rho(K)=K \subset \R$, $d(\rho)=2$ otherwise, and $|\cdot |$  denotes the usual absolute value in $\mathbb{C}$. Two distinct automorphisms $\rho_1$ and $\rho_2$ define the same valuation if and only if $\rho_1^{-1}\circ \rho_2$ is the complex conjugation. For each $\rho\in \mathcal{G}$ and for any $(u_1, \ldots, u_m, q, p)\in\mathcal{B}$, we have 
\begin{align}\label{eq2.23}
\left|\sum_{i=1}^m\alpha_i q u_i-p\right|^{d(\rho)/[K:\mathbb{Q}]}&=\left|\sum_{i=1}^m\rho(\alpha_i) q \rho(u_i)-p\right|_{\rho}.
\end{align}
For each $\mathit{v}\in  M_K^\infty$, let $\rho_\mathit{v}$  be an automorphism defining the valuation $\mathit{v}$, according to \eqref{eq3.23}, $|x|_\mathit{v}:=|x|_{\rho_\mathit{v}}$. Then the set $\{\rho_\mathit{v} : \mathit{v}\in  M_K^\infty\}$  represents the left cosets of the subgroup generated by the complex conjugation in $\mathcal{G}$.   Let $S_\ell$ (for $\ell=1,\ldots, d$), be the subset of $ M_K^\infty$  formed by those valuations $\mathit{v}$ such that $\rho_\mathit{v}|_k=\sigma_\ell: k\rightarrow \mathbb{C}$.  Note that $S_1\cup\cdots\cup S_d= M_K^\infty$. Thus, we have $ M_K^\infty = \{\rho_v : v\in  M_K^\infty\}$ 
and 
\begin{equation}\label{prodfor}
\displaystyle\sum_{\mathit{v}\in  M_K^\infty}d(\rho_\mathit{v})=[K:\mathbb{Q}].
\end{equation}
For each tuple $(u_1,\ldots, u_m, q, p)\in \mathcal{B}$, we obtain
\begin{eqnarray*}
			\prod_{\mathit{v}\in  M_K^\infty}  \left|\sum_{i=1}^m\rho_{\mathit{v}}(\alpha_i) q \rho _{\mathit{v}}(u_i)-p\right|_{\mathit{v}}= \prod_{\ell=1}^d\prod_{\mathit{v}\in S_l} \left|\sum_{i=1}^m\rho_{\mathit{v}}(\alpha_i) q \sigma_\ell(u_i)-p\right|_\mathit{v}.
		\end{eqnarray*} 
		From \eqref{eq2.23}, we see that 
		\begin{align*}
			\prod_{\mathit{v}\in  M_K^\infty}\left|\sum_{i=1}^m\rho_{\mathit{v}}(\alpha_i) q \rho_{\mathit{v}}(u_i)-p\right|_\mathit{v} &=\prod_{\mathit{v}\in  M_K^\infty} \left|\sum_{i=1}^m\alpha_i q u_i-p\right|^{d(\rho_{\mathit{v}})/[K:\mathbb{Q}]} \\
			&=  \left|\sum_{i=1}^m\alpha_i q u_i-p\right|^{{\sum_{\mathit{v}\in  M_K^\infty}d(\rho_\mathit{v})/[K}:\mathbb{Q}]}.
		\end{align*}
		Using \eqref{prodfor}, it follows that 
		\begin{equation}\label{eq2.24}
			\prod_{\ell=1}^d\prod_{\mathit{v}\in S_\ell} \left|\sum_{i=1}^m\rho_\mathit{v}(\alpha_i) q \sigma_\ell(u_i)-p\right|_\mathit{v}=
			\left|\sum_{i=1}^m\alpha_i q u_i-p\right|.
		\end{equation}
Set ${\bf x} = (x_{00}, x_{11},\ldots,x_{1m}, \ldots, x_{d1}, \ldots, x_{dm})$. Now, for each $\mathit{v}\in S$, we define $md+1$ linearly independent linear forms in the $md+1$  variables as follows. For each $\ell = 1, \ldots, d$ and for an archimedean place $\mathit{v}\in S_\ell$, we define  
\begin{equation*}
			L_{\mathit{v},00}(x_{00}, x_{11},\ldots,x_{1m}, \ldots, x_{d1}, \ldots, x_{dm})= -x_{00} +\rho_\mathit{v}(\alpha_1) x_{l1}+\cdots+\rho_\mathit{v}(\alpha_m) x_{lm}
	\end{equation*}
and for all $v\in S\backslash  M_K^\infty$, 
    \[L_{\mathit{v},00}(x_{00}, x_{11},\ldots,x_{1m}, \ldots, x_{d1}, \ldots, x_{dm})= x_{00}.\]
    Now for $v \in S$ and for $1 \leq i\leq m, 1\leq j\leq d$, define 
	$$
	L_{v, ji}(x_{00}, x_{11},\ldots,x_{1m}, \ldots, x_{d1}, \ldots, x_{dm})=x_{ji}.
	$$
	Clearly, one can see that these linear forms are $\mathbb{Q}$-linearly independent for each $v$. Let the special points $\mathbf{x} \in K^{md+1}$ be of the form
	$$
	\mathbf{x}=(p, q\sigma_1(u_1),\ldots,q\sigma_1(u_m), \ldots, q\sigma_d(u_1),\ldots,q\sigma_d(u_m)) \in K^{md+1}.
	$$
	Let $L_{\mathit{v}}$ denote the collection of all the $md+1$ linear forms corresponding to $v$.
	Now we estimate the product
	\begin{equation}\label{eq2.25}
	\prod_{\mathit{v\in S}}\prod_{L \in L_{\mathit{v}}}\frac{|L(\mathbf{x})|_\mathit{v}}{\Vert\mathbf{x}\Vert_\mathit{v}}=\frac{\left(\prod_{\mathit{v\in S}}|L_{\mathit{v},00}(\mathbf{x})|_\mathit{v}\right)\prod_{\mathit{v}\in S}\left(\prod_{j=1}^{d} \prod_{i=1}^{m}|L_{\mathit{v},ji}(\mathbf{x})|_\mathit{v}\right)}{\prod_{\mathit{v\in S}}\prod_{L \in L_{\mathit{v}}}\Vert\mathbf{x}\Vert_\mathit{v}}.
	\end{equation}
		First, we find the product for $\mathit{v}\in S$ with $i\neq 0$ and $j\neq 0$. Using the fact that $L_{\mathit{v},ji}(\mathbf{x}) =q \sigma_j(u_i)$, for all $1\leq j\leq d, 1\leq i \leq m$ and for all $v$, we obtain the following.
		\begin{align*}
			\prod_{\mathit{v}\in S}\prod_{j=1}^{d} \prod_{i=1}^{m}|L_{\mathit{v},ji}(\mathbf{x})|_\mathit{v}&=
			\prod_{\mathit{v}\in S}\prod_{j=1}^{d}\prod_{i=1}^{m}|q \sigma_j(u_i)|_\mathit{v}
			=\prod_{\mathit{v}\in S}\prod_{j=1}^{d}\prod_{i=1}^{m}|q|_\mathit{v}\prod_{j=1}^{d}\prod_{i=1}^{m}\prod_{\mathit{v}\in S}|\sigma_j(u_i)|_\mathit{v}.
		\end{align*}
		Since $\sigma_j(u_i)$  are $S$-units, by the product formula we have
		$$
		\prod_{\mathit{v}\in S}|\sigma_j(u_i)|_\mathit{v}=\prod_{\mathit{v}\in M_K}|\sigma_j(u_i)|_\mathit{v}=1.  
		$$
Consequently, from the above equality and from the formula \eqref{prodfor}, we get  
		\begin{align}
			\begin{split}\label{eq2.26}
				\prod_{\mathit{v}\in S}\prod_{j=1}^{d} \prod_{i=1}^{m}|L_{\mathit{v},ji}(\mathbf{x})|_\mathit{v}&=\prod_{\mathit{v}\in S}\prod_{j=1}^{d}\prod_{i=1}^{m}|q|_\mathit{v}\leq \prod_{v\in  M_K^\infty}\prod_{j=1}^{d}\prod_{i=1}^{m}|q|_\mathit{v}= |q|^{md}.
			\end{split}
		\end{align}
Since the coordinates of $\mathbf{x}$ are $S$-integers, i.e., $\|\mathbf{x}\|_\mathit{v}\leq 1$ for all $\mathit{v}\not\in S$, we estimate the product in the denominator of \eqref{eq2.25} as
		\begin{equation}\label{eq2.27}
			\prod_{\mathit{v\in S}}\prod_{L \in L_{\mathit{v}}}\Vert\mathbf{x}\Vert_\mathit{v} \geq \left(\prod_{\mathit{v\in M_K}}\Vert\mathbf{x}\Vert_\mathit{v}\right)^{md+1}\geq H(\mathbf{x})^{md+1}.
		\end{equation}
From\eqref{eq2.24}, \eqref{eq2.26}, \eqref{eq2.27} and using the integrality of $p$, we get $$
		\prod_{\mathit{v\in S}}\prod_{j=1}^{d}\prod_{i=1}^{m}\frac{|L_{\mathit{v},00}(\mathbf{x})|_\mathit{v}|L_{\mathit{v},ji}(\mathbf{x})|_\mathit{v}}{\|\mathbf{x}\|_\mathit{v}}\leq \frac{|q|^{md}}{H(\mathbf{x})^{md+1}} \left|\sum_{i=1}^m\alpha_i q u_i-p\right|.
		$$
Therefore, by \eqref{prop1eq1}, we get 
		$$
		\prod_{\mathit{v\in S}}\prod_{j=1}^{d}\prod_{i=1}^{m}\frac{|L_{\mathit{v},00}(\mathbf{x})|_\mathit{v}|L_{\mathit{v},ji}(\mathbf{x})|_\mathit{v}}{\|\mathbf{x}\|_\mathit{v}}\leq \frac{1}{H(\mathbf{x})^{md+1}}\frac{1}{\left(\prod_{i=1}^mH( u_i)\right)^{\varepsilon}|q|^{\varepsilon}}.
		$$
Using the fact that  every conjugate of $u_i$ has an absolute value bounded  by its Weil height power $d$, we get 
		\begin{align*}
			|p| &\leq \left|\sum_{i=1}^m \alpha_i q u_i-p +\sum_{i=1}^m \alpha_i q u_i\right| \leq \left|\sum_{i=1}^m \alpha_i q u_i\right|+1 \leq 1+ m|q| \max_{1\leq i\leq m}{|\alpha_i|}H(u_i)^{d}. 
		\end{align*}
Since $\displaystyle\max_{1\leq i\leq m}H(u_i)\to\infty$, we have
$
|p|  \leq C_1(\alpha_i, m) |q| \max_{1\leq i\leq m}H( u_i)^{d},
$
where $C_1(\alpha_i, m)>1$ is a constant that depends only on $\alpha_i$ and $m$. Notice that 
\begin{align*}
H(\mathbf{x})&=\prod_{\mathit{v}\in M_K}\mbox{max}\{|p|_\mathit{v},|q\sigma_1(u_1)|_\mathit{v},\ldots,|q\sigma_1(u_m)|_\mathit{v}, \ldots, |q\sigma_d(u_1)|_\mathit{v},\ldots,|q\sigma_d(u_m)|_\mathit{v}\}\\
&\leq \prod_{\mathit{v}\in S}\mbox{max}\{|p|_\mathit{v},|q\sigma_1(u_1)|_\mathit{v},\ldots,|q\sigma_1(u_m)|_\mathit{v}, \ldots, |q\sigma_d(u_1)|_\mathit{v},\ldots,|q\sigma_d(u_m)|_\mathit{v}\}\\
&\leq |pq| \prod_{\mathit{v}\in S}\mbox{max}\{1,|\sigma_1(u_1)|_\mathit{v},\ldots,|\sigma_1(u_m)|_\mathit{v}, \ldots, |\sigma_d(u_1)|_\mathit{v},\ldots,|\sigma_d(u_m)|_\mathit{v}\}\\ 
&\leq|pq|\left(\prod_{i=1}^m H(u_i)\right)^d.
\end{align*} 
Since $|p|  \leq C_1(\alpha_i, m) |q| \max_{1\leq i\leq m}H( u_i)^{d}$, we get    
		$$
		H(\mathbf{x}) \leq C_2|q|^2\left(\prod_{i=1}^mH(u_i)\right)^{2d} \leq C_2\left(|q|\prod_{i=1}^mH(u_i)\right)^{2d}
		$$ 
        and this implies that $ |q|\prod_{i=1}^mH(u_i) \geq C_3 H(\mathbf{x})^{1/2d}$,
		where $C_2$ and $C_3$ are positive constants depending only on $\alpha_i$, $m$ and $d$. Thus
		$$
		\prod_{\mathit{v\in S}}\prod_{j=1}^{d}\prod_{i=1}^{m}\frac{|L_{\mathit{v},00}(\mathbf{x})|_\mathit{v}|L_{\mathit{v},ji}(\mathbf{x})|_\mathit{v}}{\|\mathbf{x}\|_\mathit{v}}\leq \frac{1}{H(\mathbf{x})^{md+1+\varepsilon/2d}}.
		$$
By Theorem \ref{schli}, there exists a proper subspace of $K^{md+1}$ which contains infinitely many points $\mathbf{x}=(p, q\sigma_1(u_1),\ldots,q\sigma_1(u_m),\ldots, q\sigma_d(u_1),\ldots,q\sigma_d(u_m))$. So, we obtain a non-trivial  relation of the form
\begin{equation}\label{eq2.28}
			a_0 p+a_{11}q\sigma_1(u_1)+\cdots+a_{1m}q \sigma_1(u_m)+\cdots+ a_{d1}q\sigma_d(u_1)+\cdots+a_{dm} q \sigma_d(u_m)=0,
\end{equation} where $a_{ji}\in K$,
satisfied by all tuples $(u_1,\ldots, u_m, q, p)\in\mathcal{B}_1 \subset \mathcal{B}$ for some infinite subset $\mathcal{B}_1$  of $\mathcal{B}$. Also,  for each tuple $(u_1,\ldots, u_m, q, p) \in \mathcal{B}_1$, without loss of generality, we can assume that $|q\alpha_1 u_1| > 1$. 

Since not all $a_{ji}$'s are $0$, clearly, we can conclude that at least one among $a_{11},\ldots, a_{1m}, \ldots, a_{d1}, \ldots, a_{dm}$ is non-zero. Now, we have the following claim.
		
\bigskip

\noindent{\sc Claim: }  There exists a non-trivial relation as in \eqref{eq2.28} with $a_0 = 0$.
		
\bigskip
		
If $a_0\neq 0$, we rewrite the relation \eqref{eq2.28} as 
\begin{equation}\label{eq2.29} 
			p=-\frac{a_{11}}{a_0}q\sigma_1(u_1)-\cdots -\frac{a_{1m}}{a_0}q\sigma_1(u_m)-\cdots-\frac{a_{d1}}{a_0}q\sigma_d(u_1)-\cdots -\frac{a_{dm}}{a_0}q\sigma_d(u_m).
		\end{equation} 
Suppose that for some index $j\in \{2, \ldots, d\}$ and $i\in \{1, \ldots, m\}, \; \sigma_j(a_{1i}/a_0)\neq a_{ji}/a_0$, then we apply the automorphism $\sigma_j$ to both sides of \eqref{eq2.29} and subtract term-by-term from \eqref{eq2.29} to eliminate $p$. Since the coefficient of $\sigma_j(u_i)$ is $\sigma_j(a_{1i}/a_0)-a_{ji}/a_0\neq 0$, we obtain non-trivial relation involving only the terms \[\sigma_1(u_1), \ldots, \sigma_1(u_m), \ldots, \sigma_d(u_1), \ldots, \sigma_d(u_m).\]
Hence, we have proved the claim  in this case.
		
Now assume that $\sigma_j(a_{1i}/a_0)= a_{ji}/a_0$ for all $1\leq i\leq m$ and $1\leq j\leq d$. So in particular all coefficients $a_{ji}/a_0$ are non-zero. Setting $\lambda_i = -a_{1i}/a_0$, we rewrite \eqref{eq2.29} as follows:
		\begin{equation}{\label{eq2.29a}}
			p=q(\sigma_1(\lambda_1)\sigma_1(u_1)+\cdots+\sigma_1(\lambda_m)\sigma_1(u_m)+\cdots+\sigma_d(\lambda_1)\sigma_d(u_1)+\cdots+ \sigma_d(\lambda_m)\sigma_d(u_m)).
		\end{equation}
		At first, suppose that $\lambda_i$ does not belong to $k$. Then there exists an automorphism $\tau\in \mathcal{H}$ with $\tau(\lambda_i)\neq \lambda_i$. By applying $\tau$ on both sides of \eqref{eq2.29a} and subtracting term-by-term from \eqref{eq2.29a}, we obtain
		\begin{equation}{\label{eq2.29b}}
			\sum_{i=1}^m (\lambda_i-\tau(\lambda_i)) \sigma_1(u_i)+\sum_{j=2}^d\sum_{i=1}^m(\sigma_j(\lambda_i)\sigma_j(u_i)-\tau \circ\sigma_j(\lambda_i)\tau \circ\sigma_j(u_i))=0.
		\end{equation}
		Note that on $k$, $\tau \circ\sigma_j$ coincides with some $\sigma_i$ for some integer $i$. Since $\tau\in \mathcal{H}$  and $\sigma_2,\ldots,\sigma_d \not\in \mathcal{H}$, none of the $\tau \circ \sigma_j$ with $j\geq 2$ belongs to $\mathcal{H}$.  Hence, the relation in \eqref{eq2.29b} can be written as a linear combination of $\sigma_j(u_i)$ with the property that the coefficient of $\sigma_1(u_i)$  is $\lambda_i-\tau(\lambda_i)\ne 0$ for $1\leq i \leq m$.   Therefore, we obtain a non-trivial relation among the $\sigma_j(u_i)$, as claimed.

		Next, we may suppose that $\lambda_i\in k$ for $1\leq i\leq m$. Recall that $\sigma_1$ is the identity embedding. By adding $-\sum_{i=1}^m \alpha_iq u_i$ to both sides of \eqref{eq2.29}, we get 
		\begin{align*}
			\left|p-\sum_{i=1}^m \alpha_iq u_i\right|&=\left|\sum_{i=1}^m(\lambda_i-\alpha_i)q \sigma_1(u_i)+\sum_{j=2}^d\sum_{i=1}^mq\sigma_j(\lambda_i u_i)\right|.
		\end{align*}
Since $H({\bf u})\leq \prod_{i=1}^m H(u_i)$ and then from \eqref{prop1eq1}, we infer
		$$
		\left|\sum_{i=1}^m(\lambda_i-\alpha_i)q \sigma_1(u_i)+\sum_{j=2}^d\sum_{i=1}^mq\sigma_j(\lambda_i u_i)\right|<\frac{1}{H({\bf u})^{\varepsilon} |q|^{md+\varepsilon}}.
		$$
Dividing by $q$ on both sides,
\begin{align}\label{eq2.300}
 \begin{split}
			0<\left|\sum_{i=1}^m(\lambda_i-\alpha_i)\sigma_1(u_i)+\sum_{j=2}^d\sum_{i=1}^m\sigma_j(\lambda_i u_i)\right|&<\frac{1}{H({\bf u})^{\varepsilon} |q|^{md+1+\varepsilon}} \\&< \frac{1}{H({\bf u})^{\varepsilon}|q|^{1+\varepsilon}}.
\end{split}
\end{align}  
Put  $\sigma_1(\lambda_i)-\alpha_i =\beta_{1i}$ for $1\leq i\leq m$. Then we  re-write  \eqref{eq2.300} as   
\begin{equation}\label{eq2.30}
0<\left|\sum_{i=1}^m\beta_{1i}\sigma_1(u_i)+\sum_{j=2}^d\sum_{i=1}^m\sigma_j(\lambda_i u_i)\right|< \frac{1}{H({\bf u})^{\varepsilon}|q|^{1+\varepsilon}} 
		\end{equation} and it hold for all tuples $(u_1, \ldots, u_m, q, p) \in\mathcal{B}_1$.  In order to apply Lemma \ref{lem2.4},  we distinguish two cases, namely $\beta_{1i} =0$ for all $1\leq i\leq m$ and $\beta_{1i}\neq 0$ for some $1\leq i\leq m$. 
		
{\it First Case:} Suppose that $\beta_{1i}=0$, i.e., $\sigma_1(\lambda_i)=\alpha_i$ for $1\leq i\leq m$ and this implies \[(q\alpha_1u_1, \ldots, q\alpha_mu_m) = (q\lambda_1u_1, \ldots, q\lambda_mu_m).\]

Since by assumption the tuple $(q\alpha_1u_1, \ldots, q\alpha_m u_m)$ is not pseudo-Pisot, so we find $(q\lambda_1u_1, \ldots, q\lambda_mu_m)$ is not a pseudo-Pisot tuple. Then using the fact that $(q\lambda_1u_1, \ldots, q\lambda_mu_m)$ is not a pseudo-Pisot tuple, we deduce that either
\begin{equation}\label{eq.03.15}
\sum_{i=1}^m q\lambda_i u_i+\sum_{\beta\in P} \beta\in \mathbb{Z},
\end{equation}
or $|\beta|<1$ for all $\beta \in P$
will not hold, where
\[P:=\{\beta\in \bar{\Q}^{\times}\backslash\{q\lambda_1 u_1,\ldots,q\lambda_mu_m\}\mid  \beta = \sigma(q\lambda_i u_i) \;\mbox{for some}\;\; \sigma\in \mbox{Gal}(\bar{\Q}/\Q)\}.\]
If necessary, by grouping the terms that satisfy 
	$q\lambda_{i_1} u_{i_1}= \sigma(q\lambda_{i_2} u_{i_2})$ where $1\leq {i_1}\neq {i_2}\leq m$ and $\sigma\in \mbox{Gal}(\bar{\Q}/\Q)$,  we can reduce \eqref{eq2.29a} such that \eqref{eq.03.15} is valid. Hence, we have
		\begin{equation}\label{eqkappa}
			\max\{|\sigma(q \lambda_i u_i)|: \; \sigma \in \mbox{Gal}({\bar{\Q}/\Q)}, \sigma(q\lambda_i u_i)\neq q \lambda_i u_i ~~\mbox{for}~~1\leq i\leq m\}  \geq 1.
		\end{equation}
		Then for all $1\leq i\leq m$,
		\begin{align}\label{eq2.32}
        \begin{split}
			&\max\{|\sigma_2(u_i)|, \ldots, |\sigma_d(u_i)|\} \max\{|\sigma_2(\lambda_i)|, \ldots, |\sigma_d(\lambda_i)|\} \\&\geq \max\{|\sigma_2(\lambda_i u_i)|,\ldots, |\sigma_d(\lambda_iu_i)|\} \geq \frac{1}{|q|}.
            \end{split}
		\end{align}
Thus from  \eqref{eq2.30} and \eqref{eq2.32}, for all tuples $(u_1, \ldots, u_m, q, p)\in\mathcal{B}_1$ we get 
		\begin{align}\label{eq2.33}
		\begin{split}
\left|\sum_{j=2}^d\sum_{i=1}^m\sigma_j(\lambda_i u_i)\right|&<\max_{\substack{2\leq j\leq d \\ 1\leq i\leq m}}\{|\sigma_j(u_i)|\} \frac{\max_{\substack{2\leq j\leq d \\ 1\leq i\leq m}}\{|\sigma_j(\lambda_i)|\}}{q^{\varepsilon}H({\bf u})^{\varepsilon}}\\&<\max_{\substack{2\leq j\leq d \\ 1\leq i\leq m}}\{|\sigma_j(u_i)|\} \frac{1}{q^{\varepsilon}H({\bf u})^{\varepsilon'}}.
\end{split}	
\end{align}
Since $H({\bf u})\to\infty$, we can find an appropriate $\varepsilon'>0$ such that \eqref{eq2.33} holds.
		Therefore, by Lemma \ref{lem2.4}, with the distinguished place $\omega$ corresponding to the identity embedding, we get an infinite subset $\mathcal{B}_2 \subset \mathcal{B}_1$ such that for all tuples $(u_1, \ldots, u_m, q, p) \in \mathcal{B}_2$ 
		there exists a non-trivial relation involving $\sigma_j(u_i)$.
		
		{\it Second Case:} Suppose $\beta_{1i}\neq 0$ for some $1\leq i\leq m$, (say) $\beta_{1i_0}\neq 0$. In this case, the term $\beta_{1i_0} \sigma_1(u_{i_0})$ appears in \eqref{eq2.30}. Since   $\max_{1\leq i\leq m}|\alpha_i q\sigma_1(u_i)|>1$, we see that  
		$$
		\max_{\substack{1\leq j\leq d \\ 1\leq i\leq m}}\{|\sigma_j(u_i)|\} \geq \max\{|u_1|,\ldots,|u_m|\}>\max_{1\leq i\leq m}|\alpha_i|^{-1} |q|^{-1}
		$$
		holds for all pairs $(u_1, \ldots, u_m, q)$, where the tuples $(u_1, \ldots, u_m, q, p)$  satisfying \eqref{eq2.30}. Thus from \eqref{eq2.30}, we deduce that 
		\begin{equation}\label{eq2.33b}
			\left|\beta_{1i_0} \sigma_1(u_{i_0})+\sum_{j=2}^d\sum_{i=1}^m\sigma_j(\lambda_i u_i)\right|< \frac{\max_{\substack{1\leq j\leq d \\ 1\leq i\leq m}}\{|\sigma_j(u_i)|\}}{q^{\varepsilon}H({\bf u})^{\varepsilon''}},
		\end{equation}
		for some appropriate $\varepsilon''>0$.
		Applying Lemma \ref{lem2.4} with the distinguished place $\omega$ as in case $\beta_{1i}=0$, we conclude the same as in the first case. This  completes the proof of the claim, that is we have a relation of the form 
		\begin{equation}\label{eq2.33c}
			\sum_{j=1}^d \sum_{i=1}^m s_{ji}q\sigma_{j}(u_i) =0
		\end{equation}
		for all pairs  $(u_1,\ldots, u_m, q, p)\in\mathcal{B}_2\subset \mathcal{B}_1$ and for some $s_{ji}\in K$. 
		
		Observe that \eqref{eq2.33c} is an $S$-unit equation.  By Lemma \ref{lem2.2}, there exists an infinite subset $\mathcal{B}_3$ of  $\mathcal{B}_2$  and a non-trivial relation of the form  $a\sigma_{j_1}(u_{i_1})+b\sigma_{j_2}(u_{i_2})=0$  for some distinct integers $i_1,i_2$ and $j_1, j_2$    and $a, b\in K^\times$  satisfied by all the tuples $(u_1,\ldots, u_m, q, p)\in\mathcal{B}_3$.  
		
		\bigskip
		
		{\it Sub case I:} Suppose that $i_1=i_2 =i$ and $j_1\neq j_2$. Hence, 
		$$
		-\sigma^{-1}_{j_2}\left(\frac{a}{b}\right)(\sigma^{-1}_{j_2}\circ \sigma_{j_1})(u_{i})=u_{i}$$
		is true for all tuples $(u_1, \ldots, u_m, q, p) \in \mathcal{B}_3$.   Therefore, for any two tuples \[(u_1^{(1)}, \ldots, u_m^{(1)}, q^{(1)}, p^{(1)})\quad \mbox{and} \quad (u_1^{(2)}, \ldots, u_m^{(2)}, q^{(2)}, p^{(2)}) \in \mathcal{B}_3,\]  we have  
		$$
		\sigma^{-1}_{j_2} \circ\sigma_{j_1}(u_{i}^{(1)}/u_{i}^{(2)})=u_{i}^{(1)}/u_{i}^{(2)}.
		$$
		That is,  the element $u_{i}^{(1)}/u_{i}^{(2)}$  is fixed by the automorphism $\sigma^{-1}_{j_2} \circ\sigma_{j_1}\notin \mathcal{H}$, and hence  $u_{i}^{(1)}/u_{i}^{(2)}$ belongs to the proper subfield 
		$k'$ of $k$ which is fixed by the subgroup generated by $\mathcal{H}$ and $\sigma^{-1}_{j_2} \circ\sigma_{j_1}$.
		So, fix a nonzero $(u_1^{(1)}, \ldots, u_m^{(1)}) \in k^m$ with $(u_1^{(1)}, \ldots, u_m^{(1)}, q^{(1)},p^{(1)}) \in \mathcal{B}_3$ and take any other tuple  $(u_1,\ldots, u_m, q, p) \in \mathcal{B}_3$, then we can get  
		\[\left(\frac{u_1}{u_1^{(1)}}, \ldots, \frac{u_m}{u_m^{(1)}}\right)\in (k')^m.\]  
		
		\bigskip

		{\it Sub case II:} Suppose that $i_1\neq i_2$ and $j_1= j_2= j$. Thus, 
		$$
		-\sigma^{-1}_{j}\left(\frac{a}{b}\right)=\frac{u_{i_2}}{u_{i_1}}$$
		is true for all tuples $(u_1, \ldots, u_m, q, p) \in \mathcal{B}_3$. So, \[(u_1^{(1)}, \ldots, u_m^{(1)}, q^{(1)}, p^{(1)})\quad \mbox{and} \quad (u_1^{(2)}, \ldots, u_m^{(2)}, q^{(2)}, p^{(2)}) \in \mathcal{B}_3,\] and we have  
		\begin{equation*}
			\left(\frac{u_{i_2}^{(1)}}{u_{i_1}^{(1)}}\right)\left(\frac{u_{i_2}^{(2)}}{u_{i_1}^{(2)}}\right)^{-1}=1,
		\end{equation*}
		which is not possible by assumption (i).
		
		\bigskip

		{\it Sub case III:} Suppose that $i_1\neq i_2$ and $j_1\neq j_2$. Hence, 
		$$
		-\sigma^{-1}_{j_2}\left(\frac{a}{b}\right)(\sigma^{-1}_{j_2}\circ \sigma_{j_1})(u_{i_1})=u_{i_2}$$
		is true for all tuples $(u_1, \ldots, u_m, q, p) \in \mathcal{B}_3$. Again \[(u_1^{(1)}, \ldots, u_m^{(1)}, q^{(1)}, p^{(1)})\quad \mbox{and} \quad (u_1^{(2)}, \ldots, u_m^{(2)}, q^{(2)}, p^{(2)}) \in \mathcal{B}_3,\] and hence  
		$$
		\sigma^{-1}_{j_2} \circ\sigma_{j_1}(u_{i_1}^{(1)}/u_{i_1}^{(2)})=u_{i_2}^{(1)}/u_{i_2}^{(2)}.
		$$ This implies 
		\[\left(\sigma_{j_1}(u_{i_1}^{(1)}/u_{i_1}^{(2)})\right)\left(\sigma_{j_2}(u_{i_2}^{(1)}/u_{i_2}^{(2)})\right)^{-1}=1,\]
		which is also not true by assumption (i). This completes the proof of the Proposition \ref{prop1}. \qed

\subsection{Proof of Theorem \ref{maintheorem1}} 
	Since $\Gamma$ is a finitely generated multiplicative subgroup of $\bar{\Q}^{\times}$, by enlarging $\Gamma$ if necessary, we can reduce to the situation where $\Gamma\subset \bar{\Q}^{\times}$ is the group of $S$-units for a suitable Galois extension $K$ over $\Q$ containing $\alpha_1, \ldots,\alpha_m$ and for a suitable finite set $S$ of places of $K$ containing all the archimedean places. Also, $\Gamma$ is stable under Galois conjugation.
	
	Suppose by contradiction that there are infinitely many tuples $(u_1, \ldots, u_m, q, p)\in (\mathcal{O}^{\times }_S)^m\times \mathbb{Z}^2$ satisfying conditions (i) to (iv) in Theorem \ref{maintheorem1} together with the following inequality
	\[0< \left|\sum_{i=1}^m \alpha_iq u_i - p\right|<\frac{1}{\left(\prod_{i=1}^mH(u_i)\right)^\varepsilon |q|^{md+\varepsilon}}\leq\frac{1}{H({\bf u})^{\varepsilon} |q|^{md+\varepsilon}}.\]
    The last inequality on the right hand side follows from the fact that $H({\bf u})\leq \prod_{i=1}^m H(u_i)$, where ${\bf u}:= (u_1, \ldots, u_m)$. Then by inductively,  we construct sequences $\{\alpha_{\ell,1}\}_{\ell=0}^\infty$,$\ldots$, $\{\alpha_{\ell,m}\}_{\ell=0}^\infty$ whose elements are in $K$, an infinite decreasing chain $\mathcal{B}_i$ of an infinite subset of $\mathcal{B}$  and an infinite strictly decreasing chain $k_i$ of subfields of $K$ satisfying the following properties.
	\bigskip
	
	{\it For each integer $n\geq 0$, $\mathcal{B}_n\subset (k_n^m\times \mathbb{Z}^2)\cap\mathcal{B}_{n-1}$, $k_n\subset k_{n-1}$, $k_n\neq k_{n-1}$  and all but finitely many tuples $(u_1, \ldots, u_m, q, p)\in\mathcal{B}_n$ satisfying 
		\begin{itemize}
			\item for any two tuples $(u_1,\ldots,u_m)$ and $(u'_1,\ldots,u'_m)$, we have $\frac{u_{i_1}}{u_{i_2}}\neq \frac{u'_{i_1}}{u'_{i_2}}$ with $1\leq i_1\neq i_2\leq m$ and it is stable under Galois conjugation over $\Q$;
			\item $\displaystyle \max_{1\leq i\leq m}\{|\alpha_{0i}\cdots\alpha_{ni} qu_i|\}>1$,
			\item the tuple $(\alpha_{0,1}\cdots\alpha_{n,1} qu_1, \ldots, \alpha_{0, m}\cdots\alpha_{n, m} qu_m)$ is not a  pseudo-Pisot tuple and
			\begin{equation}\label{eq4.2}
				|\alpha_{0,1}\cdots\alpha_{n,1}qu_1+\cdots+\alpha_{0,m}\cdots\alpha_{n,m}qu_m-p|<\frac{1}{(\prod_{i=1}^mH( u_i))^{\varepsilon/(n+1)} |q|^{md+\varepsilon}}.
			\end{equation}
	\end{itemize} }
	If such sequences exist, then we eventually get a contradiction to the fact that the number field $K$ does not admit any infinite strictly decreasing chain of subfields. Therefore, in order to finish the proof of Theorem \ref{maintheorem1}, it suffices to construct such sequences.
	\smallskip
	
	We proceed with our construction by applying induction on $n$: for $n=0$, put $\alpha_{0,i}=\alpha_i$ for each integer $i= 1,\ldots, m$, $k_0 = K$ and $\mathcal{B}_0=\mathcal{B}$, and we are done in this case because of our assumption. 
	By the induction hypothesis, we assume that  $\alpha_{n, i}$, $k_n$  and $\mathcal{B}_n$ for an integer $n\geq 0$ exist and satisfy \eqref{eq4.2}. Now applying Proposition \ref{prop1} with $k=k_n$  and 
	$$
	\delta_n^{(1)}=\alpha_{0,1}\cdots\alpha_{n,1},  \ldots, \delta_n^{(m)} = \alpha_{0,m}\cdots\alpha_{n,m},
	$$
	we obtain an element $\alpha_{n+1, i}\in k_{n} $ for $1\leq i\leq m$, a proper subfield $k_{n+1}$ of $k_n$ and an infinite set $\mathcal{B}_{n+1}\subset\mathcal{B}_n$  such that all tuples $(u_1,\ldots, u_m, q, p)\in \mathcal{B}_{n+1}$  satisfy $(u_1, \ldots, u_m)=(\alpha_{n+1, 1}v_1, \ldots, \alpha_{n+1, m}v_m)$  with $(v_1, \ldots, v_m)\in k_{n+1}^m$. Note that since $(u_1, \ldots, u_m)\in(\mathcal{O}^\times_S)^m$, we observe that $(v_1, \ldots, v_m)\in(\mathcal{O}^\times_S)^m$.  Hence, as $(u_1, \ldots, u_m)$ varies, we see that $(v_1, \ldots, v_m)$ varies over $(\mathcal{O}^\times_S)^m$.  Thus,  we can assume that $(u_1, \ldots, u_m, q, p)\in\mathcal{B}_{n+1}$ if and only if $(v_1, \ldots, v_m, q, p)\in\mathcal{B}_{n+1}$. 
	
By the induction hypothesis,  for every tuple $(u_1, \ldots, u_m, q, p)\in \mathcal{B}_{n+1}$ we have $\max_{1\leq i\leq m}\{|\delta_{n}^{(i)} qu_i|\}>1$ and the tuple $(\delta_{n}^{(1)} qu_1, \ldots, \delta_{n}^{(m)}q u_m)$ is not pseudo-Pisot. Since $\delta_n^{(i)}qu_i =  \delta_n^{(i)}q\alpha_{n+1, i}v_i =  \delta_{n+1}^{(i)}qv_i $, for every tuple $(u_1, \ldots, u_m, q, p)\in \mathcal{B}_{n+1}$, we deduce that $\max_{1\leq i\leq m}\{|\delta_{n+1}^{(i)} qv_i|\} >1$ and  the tuple $(\delta_{n+1}^{(1)} qv_1, \ldots, \delta_{n+1}^{(m)}q v_m)$ is not pseudo-Pisot and 
	\begin{align}\label{eqr50}
		|\delta^{(1)}_{n+1} qv_1+\cdots+\delta^{(m)}_{n+1} qv_m-p|&<\frac{1}{(\prod_{i=1}^mH(\alpha_{n+i,1} v_i)) ^{\varepsilon/(n+1)}|q|^{md+\varepsilon}}.
	\end{align}
	We see that 
	\begin{align*}
		H(v_i) \leq H(\alpha_{n+i,1}^{-1})H(\alpha_{n+i,1} v_i),
	\end{align*} which implies 
	\begin{equation}\label{eqr29}
		\prod_{i=1}^mH(\alpha_{n+i,1} v_i)\geq  \prod_{i=1}^mH(\alpha_{n+i,1}^{-1})^{-1} H(v_i).
	\end{equation}
	If we can show that 
	\begin{equation}\label{eqr30}
\prod_{i=1}^mH(\alpha_{n+i,1}^{-1})^{-1}>\left(\prod_{i=1}^mH(v_i)\right)^{-1/n+2},\end{equation}
for almost all $v_1, \ldots, v_m\in K$, then we have 
	\begin{equation}\label{eqr51}
		\prod_{i=1}^m H(\alpha_{n+i,1} v_i)\geq \left(\prod_{i=1}^mH(v_i)\right)^{(n+1)/(n+2)},
	\end{equation}for almost all $v_1, \ldots, v_m\in K$. Suppose that \[\prod_{i=1}^mH(\alpha_{n+i,1}^{-1})^{-1}<\left(\prod_{i=1}^mH(v_i)\right)^{-1/n+2}\]
and hence
\begin{equation}\label{eqr31}
\left(\prod_{i=1}^mH(\alpha_{n+i,1}^{-1})\right)^{n+2}>\prod_{i=1}^mH(v_i),\end{equation}
	that is, $H(v_1, \ldots, v_{m})$ is bounded. So using the Northcott property we can conclude that \eqref{eqr31} is true for only finitely many tuples $(v_1, \ldots, v_{m})$. Then  \eqref{eqr30} holds for all but finitely many such tuples $(v_1, \ldots, v_m, q, p)\in\mathcal{B}_{n+1}$. Therefore, from \eqref{eqr50} and \eqref{eqr51} for all but  finitely many such tuples $(v_1, \ldots, v_m, q, p)\in\mathcal{B}_{n+1}$, we have the following inequality 
	$$
	|\delta^{(1)}_{n+1} q v_1+\cdots+\delta^{(m)}_{n+1}qv_m-p|<\frac{1}{(\prod_{i=1}^mH(v_i))^{\varepsilon/(n+2)} |q|^{md+\varepsilon}}.
	$$
	holds true. This proves the induction step and completes the proof of Theorem \ref{maintheorem1}.  \qed

	\section{Proof of Theorem \ref{thm2}}\label{sec4}
    Since $\Gamma$ is a finitely generated multiplicative subgroup of $\bar{\Q}^{\times}$. Note that if necessary, we can enlarge $\Gamma$ and reduce to the situation where $\Gamma\subset \bar{\Q}^{\times}$ is the group of $S$-units for a suitable number field $K$ of degree $d$ containing $\alpha_1, \ldots,\alpha_m$ and for a suitable finite set $S$ of places of $K$ containing all the archimedean places. Also, $\Gamma$ is stable under Galois conjugation.
    \par
	Recall that a tuple $(\beta_1, \ldots, \beta_n)$ satisfies the property
	\begin{itemize}
		\item  (P1) if $\rho \neq \beta_i$ is a Galois conjugate to $\beta_i$ over $\Q$, then  $\frac{\rho}{\beta_i} \notin \mu$;
		\item (P2) if $\beta_i \sim \beta_j$, then $\beta_i$ is Galois conjugate to $\beta_j$ over $\Q$.
	\end{itemize}
By assumption, each tuple ${\bf u}:=(u_1, \ldots, u_m) \in \mathcal{N}_1'$ satisfies both properties $(P1)$ and $(P2)$.
We can partition $\{u_1, \ldots, u_m\}$ into $\mathfrak{h_{\bf u}}$ equivalence classes having cardinalities $\mathfrak{e}_i$, that is, $\{u_{i1}, \ldots, u_{i\mathfrak{e}_i}\}$, $1 \leq i \leq \mathfrak{h_{\bf u}}$, sets in the resulting partition. In addition, we relabel $\alpha_1, \ldots, \alpha_m$ as $\alpha_{i1}, \ldots,  \alpha_{i\mathfrak{e}_i}$, $1 \leq i \leq \mathfrak{h_{\bf u}}$. Note that $\mathfrak{h_{\bf u}} \leq m$. Hence we can find an $\mathfrak{h} \in \{\mathfrak{h_{\bf u}}: {\bf u} \in \mathcal{N}_1'\}$ such that there exists an infinite set of tuples $(u_1, \ldots, u_m)$ that can be partitioned exactly into $\mathfrak{h}$ equivalence classes. Let $\mathcal{N}_{1}$ be the infinite set of tuples ${\bf u}:=(u_1, \ldots, u_m)\in \Gamma^m$, which can be partitioned exactly into $\mathfrak{h}$ equivalence classes. Now by property (P2) we have that $\{u_{i1}, \ldots, u_{i\mathfrak{e}_i}\}$, $1 \leq i \leq \mathfrak{h}$ are Galois conjugates to each other. We let $\mathfrak{d}_i \geq \mathfrak{e}_i $ denote the number of possible Galois conjugates of $u_{i1}$. We now denote $u_{i \mathfrak{e}_i+1}, \ldots u_{i \mathfrak{d}_i}$ all the other Galois conjugates of $u_{i1}$, that do not appear in $\{u_{i1}, \ldots, u_{i\mathfrak{e}_i}\}$, $1 \leq i \leq \mathfrak{h}$. Also, assume that $\sigma_i$ is the permutation on $\{1,\ldots, \mathfrak{d}_i\}$ corresponding to the action of $\sigma $ on $\{u_{i1}, \ldots, u_{i\mathfrak{d}_i}\}$, i.e., $\sigma(u_{ij})=u_{i\sigma_i(j)}$. 
By extending $K$, we may choose a finite subset $S$ of $M_K$ containing $M_K^{\infty}$ in such a way that, for each $\mathit{v} \in S$, $\sigma_{\mathit{v}}(\alpha_{ij})$ is $S$-unit for each $i$ and $j$. Let $L/ \mathbb{Q}$ be the Galois closure of $K/ \mathbb{Q}$ and let $S'$ be the set of places of $L$ lying above $S$.\\
Since $\Gamma$ is a finitely generated multiplicative group of algebraic numbers, then for any $u_{ij} \in \Gamma$, we write 
\begin{equation}\label{eqnew4.1}
u_{ij}=(\gamma_1)^{n_{ij}^{(1)}} \cdots (\gamma_s)^{n_{ij}^{(s)}},
\end{equation}
where $n_{ij}^{(t)}$ are integers for $t=1,\ldots, s$ and $\gamma_1, \ldots, \gamma_s$ are generators of $\Gamma$. Denote by 
\begin{equation}\label{eq4.1}
\mathfrak{n}=\max_{\substack{1\leq i\leq \mathfrak{h}\\ 1\leq j\leq \mathfrak{e}_i }} \{n_{ij}^{(1)}, \ldots, n_{ij}^{(s)}\}.  
\end{equation}
For the proof of Proposition \ref{prop4.1} and \ref{prop4.2}, we proceed along similar lines as in \cite[Proposition 3.4]{kul}.
\begin{proposition}\label{prop4.1}
Let $\Gamma\subset \bar{\Q}^{\times}$ be a finitely generated multiplicative group of algebraic numbers, let $\alpha_1,\ldots,\alpha_m$ be given non-zero algebraic numbers, and let $\varepsilon_1>0$ be given. Let $\mathcal{N}_{1}$ be an infinite set of tuples ${\bf u}:=(u_1, \ldots, u_m)\in \Gamma^m$, which can be partitioned exactly into $\mathfrak{h}$ equivalence classes, $|u_i|\geq 1$ for $1\leq i \leq m$, for any two tuples $(u_1,\ldots,u_m)$ and $(u'_1,\ldots,u'_m)$, we have $\frac{u_{i}}{u_j}\neq \frac{u'_i}{u'_j}$ for $1\leq i\neq j\leq m$ and satisfies both properties $(P1)$, $(P2)$ and \[0<\left\|\sum_{i=1}^m \alpha_i u_i \right\|< \frac{1}{(\prod_{i=1}^mH(u_i))^{\varepsilon_1}}.\]
If $p$ denote the nearest integer to $\sum_{i=1}^m \alpha_i u_i =\sum_{i=1}^{\mathfrak{h}} \sum_{j=1}^{\mathfrak{e}_i}  \alpha_{ij} u_{ij} $, then there exists an infinite subset $\mathcal{N}_3$ of $\mathcal{N}_{1}$ such that for all $(u_{ij})_{i,j} \in \mathcal{N}_3$, we have
\begin{equation}\label{eqn4.2}
p=\sum_{i=1}^{\mathfrak{h}} \sum_{j=1}^{\mathfrak{d}_i}  \eta_{ij} u_{ij},
\end{equation} with $\eta_{ij} \in L$ for $1 \leq i \leq \mathfrak{h}$ and $1 \leq j \leq \mathfrak{d}_i$.
\end{proposition}
\begin{proof} We can show that at least one of $H(u_i)\to \infty$ along the infinite set $\mathcal{N}_{1}$ of tuples ${\bf u}:=(u_1,\ldots,u_m)$.  We have the inequality 
\begin{equation}\label{eq.h(u)}
    H({\bf u})\leq H(u_1)\cdots H(u_m).
\end{equation} 
Suppose $H(u_i)$ is fixed for all $i \in \{1,\ldots,m\}$ for a tuple ${\bf u} \in \mathcal{N}_{1}$. Then by \eqref{eq.h(u)}, we get that $H({\bf u})$ is also fixed along $\mathcal{N}_{1}$, which is a contradiction to Northcott property. Hence we get $H(u_i)\to \infty$ for at least one $u_i$ with $i \in \{1,\ldots,m\}$ along the infinite set $\mathcal{N}_{1}$.
For $\beta \in L$, define $$|\beta|_{\mathit{v}}= |\sigma_{\mathit{v}}^{-1}(\beta)|^{d(\sigma_{\mathit{v}})/[L:\mathbb{Q]}},$$ where $\mathit{v} \in M_{L}^\infty$ corresponds to the embedding $\sigma_{\mathit{v}^{-1}}$. We assume for $1 \leq i \leq \mathfrak{h}$, $\sigma_{\mathit{v}i} $ as the induced permutation on $\{1,\ldots, \mathfrak{d}_i\}$, that is, $\sigma_{\mathit{v}}(u_{ij})=u_{i \sigma_{\mathit{v}i}(j) }$. Let $p$ be the closest integer to$\sum_{i=1}^{\mathfrak{h}} \sum_{j=1}^{\mathfrak{e}_i}  \alpha_{ij} u_{ij}$. Since $\displaystyle\sum_{\mathit{v}\in M_{L}^\infty}d(\sigma_\mathit{v})=[L:\mathbb{Q}]$, then
		\begin{align}
			\begin{split}\label{eq3.19}
				\frac{1}{(\prod_{i=1}^mH(u_i))^{\varepsilon_1}} & >  \left|\sum_{i=1}^{\mathfrak{h}} \sum_{j=1}^{\mathfrak{e}_i}  \alpha_{ij} u_{ij}-p\right|= \prod_{\mathit{v} \in M_{L}^\infty } \left|\sum_{i=1}^{\mathfrak{h}} \sum_{j=1}^{\mathfrak{e}_i}  \alpha_{ij} u_{ij}-p\right|^{d(\sigma_{\mathit{v}})/[L:\mathbb{Q}]} \\ &= \prod_{\mathit{v} \in M_{L}^\infty }\left|\sum_{i=1}^{\mathfrak{h}}\sum_{j=1}^{\mathfrak{e}_i}\sigma_{\mathit{v}}( \alpha_{ij} )u_{i \sigma_{\mathit{v}i}(j) }-p \right|_{\mathit{v}},
			\end{split}
		\end{align} for $(u_{ij})_{i,j} \in \mathcal{N}_1$.
		Let $$\mathcal{F}=\{(i,j_1,j_2): 1 \leq i \leq {\mathfrak{h}}, 1 \leq j_1 \leq \mathfrak{e}_i, 1 \leq j_2 \leq \mathfrak{d}_i\}.$$ For each $(u_{ij})_{i,j} \in  \mathcal{N}_{1}$, we associate a vector ${\bf y}:={\bf y}(u_{ij})_{i,j}$ whose components are indexed by $\mathcal{F}$ and defined by $y_{(i,j_1,j_2)}=\sigma_{\mathit{v}}( \alpha_{ij_1})u_{ij_2}$ for $(i,j_1,j_2) \in \mathcal{F} $. For $\mathit{v}\in M_{L}^\infty$ and ${\bf a}=(i,j_1,j_2) \in \mathcal{F}$, we define
		\begin{equation*}
			\delta_{\mathit{v},{\bf a}}=
			\begin{cases}
				1 & \text{if } \sigma_{\mathit{vi}}(j_1)=j_2,\\
				0 & \text{otherwise}.
			\end{cases}
		\end{equation*}
		Now we rewrite \eqref{eq3.19} as
		\begin{equation}\label{eq3.20}
			\prod_{\mathit{v} \in M_{L}^\infty}\left|\sum_{{\bf a} \in \mathcal{F}} \delta_{\mathit{v},{\bf a}}y_{\bf a}-p \right|_{\mathit{v}}<\frac{1}{(\prod_{i=1}^mH(u_i))^{\varepsilon_1}}.
		\end{equation}
		Let $\mathcal{N}_2$ be one of the infinite subsets of $\mathcal{N}_1$ such that the vector space over $L$ generated by the set $\{{\bf y}(u_{ij})_{i,j} : (u_{ij})_{i,j} \in \mathcal{N}_2 \}$ has minimal dimension. Let $V$ denote this vector space and let $r=\text{dim }V$ over $L$. By using Gaussian eliminations on the system of equations defining $V$, we obtain a subset $\mathcal{F^*}$ of $\mathcal{F}$ with $r$ elements such that $V$ is given by the linear system
		$$Y_{\bf a}=\sum_{{\bf b} \in \mathcal{F^*}} c_{{\bf a},{\bf b}} Y_{\bf b} $$ for ${\bf a} \in \mathcal{F} \backslash \mathcal{F^*}$ and $ c_{{\bf a},{\bf b}}\in L$. Hence we can write $y_{\bf a}=\sum_{{\bf b} \in \mathcal{F^*}}  c_{{\bf a},{\bf b}} y_{\bf b}$  for ${\bf a} \in \mathcal{F} \backslash \mathcal{F^*}$. So for any $\mathit{v} \in M_{L}^\infty$, 
		\[\sum_{{\bf a} \in \mathcal{F}} \delta_{\mathit{v},{\bf a}}y_{\bf a}=\sum_{{\bf b}\in \mathcal{F^*}} \bar{c}_{\mathit{v}, {\bf b} }  y_{\bf b}, \]where for ${\bf b} \in \mathcal{F^*}$, $\bar{c}_{\mathit{v}, {\bf b} } $ depends on $\mathcal{F},\mathcal{F^*}, \mathit{v}, \delta_{\mathit{v},{\bf a}}$ and $c_{{\bf a},{\bf b}}$. We rewrite \eqref{eq3.20} as,
		\begin{equation}\label{eq2.50}
			\prod_{\mathit{v} \in M_{L}^\infty}\left|\sum_{{\bf b} \in \mathcal{F}^*} \bar{c}_{\mathit{v}, {\bf b} }  y_{\bf b}-p \right|_{\mathit{v}}< \frac{1}{(\prod_{i=1}^mH(u_i))^{\varepsilon_1}}. 
		\end{equation} 
We choose the vector ${\bf X}$ whose first coordinate is $X_0$ and other coordinates $X_{\bf b}$ with ${\bf b} \in \mathcal{F^*}$.
Now we will apply the subspace theorem (Theorem \ref{schli}) to a set of $r+1$ linear forms denoted by $L_{w,0}({\bf X})$ and $L_{w,{\bf b}}({\bf X})$ for ${\bf b} \in \mathcal{F^*}$.
For $w \in S'\setminus M_{L}^{\infty}$, define 
$$L_{w,0}({\bf X})=X_0, L_{w,{\bf b}}({\bf X})=X_{\bf b}.$$
If $w \in M_L^{\infty}$, define \[L_{w,0}({\bf X})=\sum_{{\bf b} \in \mathcal{F}^*} \bar{c}_{\mathit{w}, {\bf b} } X_{\bf b}-X_0, \quad L_{w,{\bf b}}({\bf X})=X_{\bf b}.\] Clearly, we can see that these linear forms are linearly independent for every $w \in S'$.
		
For $(u_{ij})_{i,j} \in \mathcal{N}_2$, we choose the vector ${\bf x}$ whose coordinates are given by $x_0=p$ and $x_{\bf b}=y_{\bf b}$, with ${\bf b} \in \mathcal{F^*}$. In order to apply Theorem \ref{schli}, we need to calculate the following quantity
\[\prod_{w\in S'} \left(\frac{|L_{w,0}({\bf x})|_w}{\|{\bf x}\|_w}\prod_{{\bf b} \in \mathcal{F}^*} \frac{|L_{w,{\bf b}}({\bf x})|_w}{\|{\bf x}\|_w}\right).\]
		
For each ${\bf b} \in \mathcal{F}^*$, we define $L_{w,{\bf b}}({\bf x})=y_{\bf b}$. Using the fact that for $\mathit{v} \in S$, $\sigma_{\mathit{v}}(\alpha_{ij})$ and $u_{ij}$ are $S$-units for every pair $(i,j)$ and by the product formula
		
		\begin{equation}\label{eq.3.4}
			\prod_{w\in S'}\prod_{{\bf b} \in \mathcal{F}^*} {|L_{w,{\bf b}}({\bf x})|_w}=\prod_{w\in S'}\prod_{{(i,j_1,j_2) \in \mathcal{F}^*}}|\sigma_{w}( \alpha_{ij_1})u_{ij_2}|_w=1.
		\end{equation}
		By inequality \eqref{eq2.50}, \eqref{eq.3.4} and the integrality of $p$, we get
		\begin{align}\label{eq2.60}
			\begin{split}
				&\prod_{w\in S'} \left(\frac{|L_{w,0}({\bf x})|_w}{\|{\bf x}\|_w}\prod_{{\bf b} \in \mathcal{F}^*} \frac{|L_{w,{\bf b}}({\bf x})|_w}{\|{\bf x}\|_w}\right)\\
				&=\frac{\left(\prod_{w \in S' \cap M_{L}^0} |p|_w \right)\left( \prod_{w \in S'}\prod_{{\bf b} \in \mathcal{F^*}}|y_{\bf b}|_w  \right) \left( \prod_{w \in S' \cap M_{L}^\infty}\left|\sum_{{\bf b} \in \mathcal{F}^*} \bar{c}_{\mathit{w}, {\bf b} } y_{\bf b}-p\right|_w\right)}{(\prod_{w \in S'}\|{\bf x}\|_w)^{r+1}}\\&< \frac{1}{(\prod_{w \in S'}\|{\bf x}\|_w)^{r+1}(\prod_{i=1}^mH(u_i))^{\varepsilon_1}}.
			\end{split}
		\end{align}
		Also, $\prod_{w \in S'}\|{\bf x}\|_w=H({\bf x})$. From \eqref{eq3.19},  
		$$|p| \leq  \left|\sum_{i=1}^{\mathfrak{h}} \sum_{j=1}^{\mathfrak{e}_i}  \alpha_{ij} u_{ij}\right|+1 \leq 1+ m \max_{\substack{1\leq i\leq \mathfrak{h}\\ 1\leq j\leq \mathfrak{e}_i }}|\alpha_{ij}||u_{ij}|\leq 1+ m \max_{\substack{1\leq i\leq \mathfrak{h}\\ 1\leq j\leq \mathfrak{e}_i }}|\alpha_{ij}|H(u_{ij})^d,$$ 
and this implies 
\begin{equation}\label{eqnew101}
|p| \leq c_3(\alpha_{ij}, m)\max_{\substack{1\leq i\leq \mathfrak{h}\\ 1\leq j\leq \mathfrak{e}_i }}H(u_{ij})^{d}.
\end{equation}
Now we estimate $H({\bf x})$: 
\begin{align*}\label{eqnew102}
			\begin{split}
				H({\bf x})&=\prod_{w\in  M_L} \max\{|p|_w, |y_{{\bf b}}|_w: {\bf b} \in \mathcal{F}^*\} \leq |p| \prod_{w\in  M_L} \max\{1, |y_{{\bf b}}|_w: {\bf b} \in \mathcal{F}^*\}\\& \leq  |p| \prod_{w\in S'} \max_{\substack{1 \leq i \leq \mathfrak{h}, \\ 1 \leq j \leq \mathfrak{d}_i}}\{1, |\sigma_{w}( \alpha_{ij})|_w\}  \prod_{w\in S'} \max_{\substack{1 \leq i \leq \mathfrak{h}, \\ 1 \leq j \leq \mathfrak{d}_i}}\{1, |u_{ij}|_w\} =|p|\prod_{i,j}H(\alpha_{ij})\prod_{i,j}H(u_{ij}).\end{split}             
\end{align*} 
From \eqref{eqnew101}, we get
\begin{align*}
  H({\bf x})& < c_4(\alpha_{ij}, m)\Big(\prod_{i,j}H(u_{ij})\Big)^{2d}   
\end{align*}
Let $L_w$ denote the collection of all $r+1$ linear forms corresponding to $w$.
		Using \eqref{eq2.60} we get  
\begin{equation}\label{eq.04.09}
			\prod_{w \in S'} \prod_{L \in L_w} \frac{|L({\bf x})|_w}{\|{\bf x}\|_w}<\frac{1}{H({\bf x})^{r+1}(\prod_{i, j}H(u_{ij}))^{\varepsilon_1}}. 
\end{equation}
Hence \eqref{eq.04.09} implies that,
		$$\prod_{w \in S'} \prod_{L \in L_w} \frac{|L({\bf x})|_w}{\|{\bf x}\|_w}<\frac{1}{H({\bf x})^{r+1}(\prod_{i, j}H(u_{ij}))^{\varepsilon_1}}<H({\bf x})^{-r-1-\varepsilon_1/2d}.$$ 
Thus, by Theorem \ref{schli}, we conclude that there exists an infinite subset $\mathcal{N}_3$ of $\mathcal{N}_1$ such that,
		\begin{equation}
			a_0p+\sum_{{\bf b} \in \mathcal{F^*}} a_{\bf b} y_{\bf b}=0,
		\end{equation} with $a_0, a_{\bf b} \in L$. 
		First, observe that $a_0 \neq 0$. If $a_0 = 0$, then the non-trivial relation contradicts the minimality of $r$. Therefore,  $p=\sum_{{\bf b} \in \mathcal{F^*}}\frac{-a_{\bf b}}{a_0}y_{\bf b}$ and hence we deduce equation \eqref{eqn4.2}.
	\end{proof}
	
	\begin{proposition}\label{prop4.2}
		Let $\Gamma$, $\alpha_1,\ldots,\alpha_m$, $\varepsilon_1$, $\mathcal{N}_{1}$ and $p$ be as in Proposition \ref{prop4.1}. Then there exists an infinite subset $\mathcal{N}_{6}$ of $\mathcal{N}_{1}$ such that for every $(u_{ij})_{i,j} \in  \mathcal{N}_{6}$, the following holds.
		\begin{enumerate}[label=\textnormal{(\roman*)}]
			\item  For every $\sigma\in \mbox{Gal}(L/\mathbb{Q})$ and $1 \leq i \leq \mathfrak{h}$ , let $\sigma_{i}$ denote the induced permutation on $\{1,\ldots, \mathfrak{d}_i\}$. Then, $\sigma(\eta_{ij})=\eta_{i \sigma_{i}(j)}$ for $1 \leq i \leq \mathfrak{h}$ and $1\leq j \leq \mathfrak{d}_i$.
			\item $\alpha_{ij}=\eta_{ij}$ with $1 \leq i \leq \mathfrak{h}$, $1 \leq j \leq \mathfrak{e}_i$.
			\item  Let 
			$P$ be the collection of all $\beta \notin \{\alpha_{ij}u_{ij}: 1 \leq i \leq \mathfrak{h}, 1 \leq j \leq \mathfrak{e}_i \}$ and  $\beta \text{ is Galois conjugate to }\alpha_{ij}u_{ij}$ for some $1 \leq i \leq \mathfrak{h}, 1 \leq j \leq \mathfrak{e}_i$, then
			the elements of $P$ are exactly the elements $\eta_{ij}u_{ij}$ for $ 1 \leq i \leq \mathfrak{h}, \mathfrak{e}_i < j \leq \mathfrak{d}_i$.
		\end{enumerate}
	\end{proposition}
	\begin{proof}
From equation \eqref{eqn4.2}, we have 
\begin{equation*}\label{eq.4.9}
			p=\sum_{i=1}^{\mathfrak{h}} \sum_{j=1}^{\mathfrak{d}_i}  \eta_{ij} u_{ij},
\end{equation*}
with $\eta_{ij} \in L$ for $1 \leq i \leq \mathfrak{h}$ and $1 \leq j \leq \mathfrak{d}_i$. Define $\alpha_{ij}=0$ for $1\leq i\leq \mathfrak{h}, \mathfrak{e}_i <j \leq \mathfrak{d}_i$. For every $(u_{ij})_{i,j} \in  \mathcal{N}_{1}$, 
		$$
		\left\|\sum_{i=1}^m \alpha_i u_i \right\|< \frac{1}{(\prod_{i=1}^mH(u_i))^{\varepsilon_1}},
		$$ and this implies, 
		$$
		\left|\sum_{i=1}^{\mathfrak{h}} \sum_{j=1}^{\mathfrak{e}_i}  \alpha_{ij} u_{ij}-p\right|=\left|\sum_{i=1}^{\mathfrak{h}} \sum_{j=1}^{\mathfrak{d}_i}  (\alpha_{ij}-\eta_{ij}) u_{ij}\right|< \frac{1}{(\prod_{i=1}^mH(u_i))^{\varepsilon_1}}.
		$$ 
		Let $\sigma \in \mbox{Gal}({L}/\Q)$. Then
		\begin{align*}\label{eq.2.9}
			\begin{split}
				\sigma(p)&=p=\sum_{i=1}^{\mathfrak{h}} \sum_{j=1}^{\mathfrak{d}_i}  \sigma(\eta_{ij}) \sigma(u_{ij})=\sum_{i=1}^{\mathfrak{h}} \sum_{j=1}^{\mathfrak{d}_i}  \sigma(\eta_{ij}) u_{i \sigma(j)}=\sum_{i=1}^{\mathfrak{h}} \sum_{j=1}^{\mathfrak{d}_i}  \sigma(\eta_{i \sigma^{-1}_{ i}(j)}) u_{ij}.
			\end{split}
		\end{align*}
		Using \eqref{eqn4.2}, we get
		$$\sum_{i=1}^{\mathfrak{h}} \sum_{j=1}^{\mathfrak{d}_i}(\eta_{ij}- \sigma(\eta_{i \sigma^{-1}_{ i}(j)})) u_{ij}=0.$$ 
		Let $\mathcal{N}_{4}$ be the collection of all elements $(u_{ij})_{i,j}$ with $1 \leq i \leq \mathfrak{h}$, $1\leq j \leq \mathfrak{d}_i$ and $\eta_{ij} \neq \sigma(\eta_{i \sigma^{-1}_{ i}(j)})$. 
        By Lemma \ref{lem2.2} together with assumption that for any two tuples $(u_1,\ldots,u_m)$ and $(u'_1,\ldots,u'_m)$, we have $\frac{u_{i}}{u_j}\neq \frac{u'_i}{u'_j}$ for $1\leq i\neq j\leq m$, we will get a contradiction. Hence, there exists an infinite subset of $\mathcal{N}_{4}$ such that $\eta_{ij} =\sigma(\eta_{i \sigma^{-1}_{i}(j)})$. This completes the proof of part $(i)$.
		\par
		Let $\mathcal{N}_{5}$ be the collection of $(u_{ij})_{i,j}$ such that  $\alpha_{ij} \neq \eta_{ij}$ for $1 \leq i \leq \mathfrak{h}$, $1 \leq j \leq \mathfrak{e}_i$. We will prove that $\mathcal{N}_{5}$ is finite. Assume that $\mathcal{N}_{5}$ is infinite. Let $$\mathcal{T}=\{(i,j): 1 \leq i \leq \mathfrak{h}, 1 \leq j \leq \mathfrak{d}_i, \alpha_{ij} \neq \eta_{ij}\},$$ for infinitely many $(u_{ij})_{i,j} \in  \mathcal{N}_{5}$.  Let ${\bf x}$ denote the vector having coordinates $u_{ij}$ with $(i,j) \in \mathcal{T}$. By using \eqref{eqnew4.1} and \eqref{eq4.1}, we can find a constant $c_5>1$ depending only on the generators $\gamma_1, \ldots, \gamma_s$ and $K$ such that  $H({\bf x})<c_5^{\mathfrak{n}}$. Since $|u_{ij}|>1$, we get
		\begin{equation}
			\left|\sum_{(i,j) \in \mathcal{T} }(\alpha_{ij}-\eta_{ij})u_{ij}\right|< \frac{1}{(\prod_{i=1}^mH(u_i))^{\varepsilon_1}}< \frac{\max\{|u_{ij}|: (i,j) \in \mathcal{T}\}}{H(\bf {\bf x})^{\varepsilon}},
		\end{equation}
		for some $\varepsilon>0$ and for infinitely many $(u_{ij})_{i,j} \in  \mathcal{N}_{5}$. 
		Proceeding as before by  Lemma \ref{lem2.2} together with the hypothesis that for any two tuples $(u_1,\ldots,u_m)$ and $(u'_1,\ldots,u'_m)$, we have $\frac{u_i}{u_j}\neq \frac{u'_i}{u'_j}$, we will get a contradiction. Hence $\mathcal{N}_{5}$ is finite. Therefore, $\alpha_{ij}=\eta_{ij}$ with $1 \leq i \leq \mathfrak{h}$, $1 \leq j \leq \mathfrak{e}_i$ for infinitely many elements $(u_{ij})_{i,j}$. This completes the proof of part $(ii)$.
		\par
		For  $ 1 \leq i \leq \mathfrak{h}, \mathfrak{e}_i < j \leq \mathfrak{d}_i$, there is a $\sigma \in \mbox{Gal}({L}/\Q) $ such that $\sigma(u_{i1})=u_{ij};$ thus $\sigma_{ i}(1)=j.$ Hence 
		$$\sigma(\alpha_{i1}u_{i1})=\sigma(\eta_{i1}u_{i1})=\eta_{ij}u_{ij}.$$ 
		So $\eta_{ij}u_{ij} \in P$. 
		\par 
		Conversely, if $\beta \in P$, then $\beta=\sigma(\alpha_{ij}u_{ij})$ for some $1 \leq i \leq \mathfrak{h}, 1 \leq j \leq \mathfrak{e}_i$. Then $\beta=\eta_{i \sigma_{ i}(j)}u_{i \sigma_{ i}(j)}$. Thus, we must have $\mathfrak{e}_i <\sigma_{ i}(j)\leq \mathfrak{d}_i$. Also $\eta_{ij}u_{ij}$'s are distinct for infinitely many elements $(u_{ij})_{i,j}$.
		Now let $\mathcal{N}_{6}$ be the infinite subset of $\mathcal{N}_{1}$ satisfying all the above three properties. This finishes the proof of the proposition.
	\end{proof}
\subsection{Proof of Part (i) of Theorem \ref{thm2}}
Assume the contrary that (i) does not hold. Then for any infinite set $\mathcal{N}_{1}'$ of tuples $(u_1, \ldots, u_m)\in \Gamma^m$, there exists a non-archimedean place $\nu$ such that for every positive constant $C$, we have 
\[\max\{|u_1|_v,\ldots,|u_m|_v\}\geq C\]
for 
tuple $(u_1, \ldots, u_m)$. We fix $C> 1$. If we choose a constant $C_1>\max\{|u_1|_v,\ldots,|u_m|_v\}\geq C$, then we can find a new tuple $(u_1', \ldots, u_m')\in \mathcal{N}_{1}'$ such that for the same place $\nu$,
\[\max\{|u_1'|_v,\ldots,|u_m'|_v\}\geq C_1.\]
Now we choose $C_2>\max\{|u_1'|_v,\ldots,|u_m'|_v\}$ and obtain a new tuple. Continuing in this way, we get an infinite collection of tuples $(u_1, \ldots, u_m)$ satisfying 
\[\max\{|u_1|_v,\ldots,|u_m|_v\}\geq C\]
for some fixed $C$. 
Let $p$ be the nearest integer to $\sum_{i=1}^m \alpha_i u_i $. Then using Proposition \ref{prop4.1}, there exists an infinite subset $\mathcal{N}_3$ of $\mathcal{N}_1$ such that for all $(u_{ij})_{i,j} \in \mathcal{N}_3$, we have
$$|p|_w=\left|\sum_{i=1}^{\mathfrak{h}} \sum_{j=1}^{\mathfrak{d}_i}  \eta_{ij} u_{ij}\right|_w \leq 1.$$ 
For every $1< \kappa< C^{1/\mathfrak{n}}$, we have 
$$\max\{|u_{ij}|_w: 1 \leq i \leq \mathfrak{h}, 1 \leq j\leq \mathfrak{d}_i\} \geq \kappa^{\mathfrak{n}},$$ 
where $\mathfrak{n}$ is defined as in \eqref{eq4.1}.  Let ${\bf x}$ denote the vector having coordinates $u_{ij}$ with $1 \leq i \leq \mathfrak{h}, 1 \leq j\leq \mathfrak{d}_i$.
From \eqref{eqnew4.1} and \eqref{eq4.1}, we get 
\[H({\bf x})\leq (H(\gamma_1)\cdots H(\gamma_s))^{\mathfrak{n}\mathfrak{h}\mathfrak{d}_i}.\]
Choose $0< \varepsilon<\frac{\log \kappa}{\mathfrak{h}\mathfrak{d}_i\log (H(\gamma_1)\cdots H(\gamma_s))}$, so that
\[\kappa^\mathfrak{n}>(H(\gamma_1)\cdots H(\gamma_s))^{\mathfrak{n}\mathfrak{h}\mathfrak{d}_i\varepsilon}.\]
Hence we get 
	$$|p|_w \leq 1< \frac{{|u_{11}|_w}}{\kappa^{\mathfrak{n}}}<\frac{\max\{|u_{ij}|_w: 1 \leq i \leq \mathfrak{h}, 1 \leq j\leq \mathfrak{d}_i\}}{H({\bf x})^{\varepsilon}}.$$
	holds.
	Then by Lemma \ref{lem2.4}, we get a non-trivial linear relation among  $u_{ij}$ for infinitely many $(u_{ij})_{ij}$ (with $1 \leq i \leq \mathfrak{h}, 1 \leq j\leq \mathfrak{d}_i$). By Lemma \ref{lem2.2}, there exists an infinite subset $\mathcal{N}_4$ of $\mathcal{N}_3$ and a non-trivial relation of the form 
    $$
     a u_{i,j}+b u_{i',j'}=0
    $$
    holds for all $(u_{i,j}; 1 \leq i \leq \mathfrak{h}, 1 \leq j\leq \mathfrak{d}_i)$ along the tuples $(u_1, \ldots, u_m, q, p) \in \mathcal{N}_4$. So for the tuples  \[(u_{i,j}^{(1)}; 1 \leq i \leq \mathfrak{h}, 1 \leq j\leq \mathfrak{d}_i)\quad \mbox{and} \quad (u_{i,j}^{(2)}; 1 \leq i \leq \mathfrak{h}, 1 \leq j\leq \mathfrak{d}_i)\] we have   
\begin{equation*}
\left(\frac{u_{i,j}^{(1)}}{u_{i',j'}^{(1)}}\right)\left(\frac{u_{i,j}^{(2)}}{u_{i',j'}^{(2)}}\right)^{-1}=1,
\end{equation*}
which contradicts the hypothesis that for any two tuples $(u_1,\ldots,u_m)$ and $(u'_1,\ldots,u'_m)$, we have $\frac{u_i}{u_j}\neq \frac{u'_i}{u'_j}$. Hence, our claim.
	
\subsection{Proof of Part (ii) of Theorem \ref{thm2}} 
Assume that
$\frac{\sigma(u_i)}{u_j} \notin \mu $ for $i,j=1,\ldots, m$. Then we have $\sigma(u_i)=u_{i'j'}$, for some $i'$ and $j'$ such that $1 \leq i' \leq \mathfrak{h}$, $\mathfrak{e}_i'+1 \leq j' \leq \mathfrak{d}_{i'}$.
So we need to prove $|u_{i'j'}|<1$ for $1 \leq i' \leq \mathfrak{h}$, $\mathfrak{e}_i'+1 \leq j' \leq \mathfrak{d}_{i'}$.
	
Assume that for some $1 \leq i' \leq \mathfrak{h}$, $\mathfrak{e}_i'+1 \leq j' \leq \mathfrak{d}_{i'}$, $|u_{i'j'}|> 1 $. If $p$ denotes the nearest integer to $\sum_{i=1}^{\mathfrak{h}} \sum_{j=1}^{\mathfrak{e}_i}  \alpha_{ij} u_{ij}$, then by Proposition \ref{prop4.1}, there exists an infinite subset $\mathcal{N}_3$ of $\mathcal{N}_1$ such that for all $(u_{ij})_{i,j} \in \mathcal{N}_3$, we have
	
\begin{equation}\label{eq.2.8}
p=\sum_{i=1}^{\mathfrak{h}} \sum_{j=1}^{\mathfrak{d}_i}  \eta_{ij} u_{ij}.
\end{equation} 
So, for all $(u_{ij})_{i,j} \in \mathcal{N}_3$,
$$ \left|\sum_{i=1}^{\mathfrak{h}} \sum_{j=1}^{\mathfrak{e}_i}  \alpha_{ij} u_{ij}-p\right|=\left|\sum_{i=1}^{\mathfrak{h}} \sum_{j=\mathfrak{e}_i+1}^{\mathfrak{d}_i}  \eta_{ij} u_{ij}\right|<\frac{1}{(\prod_{i=1}^mH( u_i))^{\varepsilon_1}}.$$
Choose ${\bf x}$ as the vector having  coordinates $u_{ij}$ with $1 \leq i\leq \mathfrak{h}$ and $\mathfrak{e}_i+1 \leq j\leq \mathfrak{d}_i$. Since $u_{ij}$'s are finitely generated, so we can find a constant $c_6>1$ depending only on the generators $\gamma_1, \ldots, \gamma_s$ and $K$ such that 
$$H({\bf x})<c_6^{\mathfrak{n}},$$ where $\mathfrak{n}$ is given in \eqref{eq4.1}. Now choose $0<\varepsilon< \frac{m\epsilon_1 \log (H(\gamma_1)\cdots H(\gamma_s))}{\log c_6}$ so that $\frac{1}{(\prod_{i=1}^mH(u_{i}))^{\varepsilon_1}} <{c_6^{-\mathfrak{n}\varepsilon}}.$ Hence we get 
$$\frac{1}{(\prod_{i=1}^mH( u_i))^{\varepsilon_1}}< \frac{\max\{|u_{ij}|: 1 \leq i \leq \mathfrak{h}, \mathfrak{e}_i+1 \leq j\leq \mathfrak{d}_i\}}{H({\bf x})^{\varepsilon}}.$$ Lemma \ref{lem2.4} gives a non-trivial relation among $u_{ij}$'s for infinitely many elements 
$(u_{ij})_{i,j}$. Lemma \ref{lem2.2} together with the hypothesis that the for any two tuples $(u_1,\ldots,u_m)$ and $(u'_1,\ldots,u'_m)$, we have $\frac{u_i}{u_j}\neq \frac{u'_i}{u'_j}$, lead to  a contradiction.  Hence $|u_{i'j'}| <1$ for $1 \leq i' \leq \mathfrak{h}$ and $\mathfrak{e}_i'+1 \leq j' \leq \mathfrak{d}_{i'}$. 
\subsection{Proof of Part (iii) of Theorem \ref{thm2}}  
Now we will prove that, the tuple $(\alpha_1 u_1,\ldots, \alpha_m u_m)$ is pseudo-Pisot for all but finitely many elements of $\mathcal{N}_{1}$. Assume there is an infinite subset $\mathcal{N}_{7}$ of $\mathcal{N}_{1}$ such that the tuple $(\alpha_1 u_1,\ldots, \alpha_m u_m)$ is not pseudo-Pisot for every $(u_{ij})_{i,j} \in  \mathcal{N}_{7}$.
If $p$ denotes the nearest integer to $\sum_{i=1}^{\mathfrak{h}} \sum_{j=1}^{\mathfrak{e}_i}  \alpha_{ij} u_{ij}$, then there exists an infinite subset $\mathcal{N}_3$ of $\mathcal{N}_1$ such that for all $(u_{ij})_{i,j} \in \mathcal{N}_3$, we have
$$
p=\sum_{i=1}^{\mathfrak{h}} \sum_{j=1}^{\mathfrak{d}_i}  \eta_{ij} u_{ij}.
$$
Also from \eqref{eq.1.01}, 
\begin{equation}\label{eq.04.15}
\left|\sum_{i=1}^{\mathfrak{h}} \sum_{j=1}^{\mathfrak{e}_i}  \alpha_{ij} u_{ij}-p\right|< \frac{1}{(\prod_{i=1}^mH(u_{i}))^{\varepsilon_1}}.    
\end{equation}
Let  $\mathcal{N}_{6}$ be an infinite subset of $\mathcal{N}_{7}$ such that for all elements of $\mathcal{N}_{6}$, Proposition (\ref{prop4.2}) holds.
Substituting the value of $p$ in \eqref{eq.04.15} and using Proposition \ref{prop4.2}(ii), we get
$$ \left|\sum_{i=1}^{\mathfrak{h}} \sum_{j=\mathfrak{e}_i+1}^{\mathfrak{d}_i}  \eta_{ij} u_{ij}\right|< \frac{1}{(\prod_{i=1}^mH(u_{i}))^{\varepsilon_1}}.$$ Now we claim that
$$\max\{|u_{ij}|: 1 \leq i \leq \mathfrak{h}, \mathfrak{e}_i+1 \leq j \leq \mathfrak{d}_i\}<1,$$ for all but finitely many  $(u_{ij})_{i,j} \in  \mathcal{N}_{6}$. To prove this, assume $$\max\{|u_{ij}|: 1 \leq i \leq \mathfrak{h}, \mathfrak{e}_i+1 \leq j \leq \mathfrak{d}_i\} \geq 1,$$ for infinitely many $(u_{ij})_{i,j} \in  \mathcal{N}_{6}$. 
As in the proof of part(ii), if we choose ${\bf x}$ to be the vector having  coordinates $u_{ij}$ with $1 \leq i\leq \mathfrak{h}$ and $\mathfrak{e}_i+1 \leq j\leq \mathfrak{d}_i$, we get $H({\bf x})<c_6^{\mathfrak{n}}$.
Proceeding as above, we can find an $\varepsilon$ such that,
$$\frac{1}{(\prod_{i=1}^mH(u_{i}))^{\varepsilon_1}} <{c_6^{-\mathfrak{n}\varepsilon}}.$$
Thus
$$ \left|\sum_{i=1}^{\mathfrak{h}} \sum_{j=\mathfrak{e}_i+1}^{\mathfrak{d}_i}  \eta_{ij} u_{ij}\right|< \frac{1}{(\prod_{i=1}^mH(u_{i}))^{\varepsilon_1}}<\frac{\max\{|u_{ij}|: 1 \leq i \leq \mathfrak{h}, \mathfrak{e}_i+1 \leq j \leq \mathfrak{d}_i\} }{H({\bf x})^{\varepsilon}}.$$
Using Lemma \ref{lem2.4}, we get a non-trivial relation among $u_{ij}$, for infinitely many elements $(u_{ij})_{i,j} \in  \mathcal{N}_{6}$. Using Lemma \ref{lem2.2} and together with the hypothesis that for any two tuples $(u_1,\ldots,u_m)$ and $(u'_1,\ldots,u'_m)$, we have $\frac{u_i}{u_j}\neq \frac{u'_i}{u'_j}$, we will get a contradiction. So, $\max\{|u_{ij}|: 1 \leq i \leq \mathfrak{h}, \mathfrak{e}_i+1 \leq j \leq \mathfrak{d}_i\}<1$. Thus, using the above claim and Proposition \ref{prop4.2}(iii), we conclude that, $(\alpha_1 u_1,\ldots, \alpha_m u_m)$ is pseudo-Pisot for infinitely many $(u_{ij})_{i,j} \in  \mathcal{N}_{1}$. This contradicts the choice of $\mathcal{N}_{7}$. This completes the proof of part (iii).
\subsection{Proof of Part (iv) of Theorem \ref{thm2}}
Suppose $\frac{\sigma(u_{i_1j_1})}{u_{i_2j_2} }\in \mu$, for $1\leq i_1, i_2 \leq \mathfrak{h}, 1\leq j_1 \leq \mathfrak{e_{i_1}}, 1\leq j_2 \leq \mathfrak{e_{i_2}}$. Our claim is to prove
$$\sigma(\alpha_{i_1j_1}u_{i_1j_1})=\alpha_{i_2j_2}u_{i_2j_2},$$
for all but finitely many $(u_{ij})_{ij}\in \mathcal{N}_1$, for $1\leq i \leq \mathfrak{h}, 1\leq j \leq \mathfrak{e_i}$. Suppose that there exists an infinite subset $\mathcal{N}_{8}$  of $\mathcal{N}_1$ such that $\sigma(\alpha_{i_1j_1}u_{i_1j_1}) \neq \alpha_{i_2j_2}u_{i_2j_2},$
for all $(u_{ij})_{ij}\in \mathcal{N}_{8}$. Since $\sigma(u_{i_1j_1})/u_{i_2j_2} \in \mu$, then $u_{i_1j_1} \sim u_{i_2j_2}$. Now property (P2) implies that $\rho(u_{i_1j_1})=u_{i_2j_2}$ for some $\rho \in \mbox{Gal}(\bar{\mathbb{Q}}/\mathbb{Q})$. Thus, we have $\sigma(u_{i_1j_1})/\rho(u_{i_1j_1}) \in \mu$. That is, 
$$\frac{(\rho^{-1}\circ \sigma)(u_{i_1j_1})}{u_{i_1j_1}} \in \mu.$$
Hence, by property (P1), we get $\sigma(u_{i_1j_1})=\rho(u_{i_1j_1})=u_{i_2j_2}$. But we know that,
$\sigma(u_{i_1j_1})=u_{i_1\sigma_{i_1}(j_1)}$.
Thus $u_{i_1\sigma_{i_1}(j_1)}=u_{i_2j_2}$. This implies $i_1=i_2$ and $\sigma_{i_1}(j_1)=j_2$. Let $\mathcal{N}_{6}$ be an infinite subset of $\mathcal{N}_{8}$, which satisfies Proposition \ref{prop4.2}. So, by property $(ii)$ of Proposition \ref{prop4.2}, we get 
\begin{equation*}
\begin{split}
\sigma(\alpha_{i_1j_1}u_{i_1j_1})=\sigma(\eta_{i_1j_1}u_{i_1j_1})=\eta_{i_1\sigma_{i_1}(j_1)}u_{i_1\sigma_{i_1}(j_1)}=\eta_{i_2j_2}u_{i_2j_2}=\alpha_{i_2j_2}u_{i_2j_2}.    
		\end{split}
	\end{equation*}
	But the elements of the set $\mathcal{N}_{8}$ do not satisfy this equation. Hence, our claim. 
	\smallskip
For the moreover part, we argue as follows: let $\sigma\in \mbox{Gal}(\overline{\mathbb{Q}}/\mathbb{Q})$ and a pair $(i,j)\in \{1,\ldots,m\}^2$ such that $\sigma(\alpha_i u_i)=\alpha_j u_j$ holds for all but finitely many $(u_i, u_j)$ along the tuples $(u_1,\ldots,u_m)\in\mathcal{N}_1$. Let $(u_1,\ldots,u_m)$ and $(u_1',\ldots,u'_m)$ be two tuples, then from the  relation  $\frac{\sigma(u_i)}{u_j}=\frac{\alpha_j}{\sigma(\alpha_i)}$ for all but many tuples $(u_1,\ldots,u_m)$, we conclude $\frac{\sigma(u_i)}{u_j}=\frac{\sigma(u'_i)}{u'_j}$. This completes the proof of Theorem \ref{thm2}.
	\qed
	
	\section*{Acknowledgments} The research of the first author is supported by a UGC fellowship (Ref No. 221610077314). This work was completed when the first author visited IIT Dharwad in April 2025, and she thanks the people of this institute for their hospitality and support. The second author's research is supported by NBHM research grant and ANRF-PMECRG grant ANRF/ECRG/2024/002315/PMS. S.S.R. is supported by a grant from the Anusandhan National Research Foundation (File No.:CRG/2022/000268) while working on this project.

\end{document}